\documentclass[review,onefignum,onetabnum,12pt]{siamart190516}
\nolinenumbers

\usepackage{lipsum}
\usepackage{amsfonts}
\usepackage{amssymb}
\usepackage{graphicx}
\usepackage{epstopdf}
\usepackage{algorithmic}
\usepackage{subfigure}	
\usepackage{multirow}
\usepackage{multicol}
\usepackage{float}
\usepackage{CJK}
\usepackage{lineno}
\ifpdf
  \DeclareGraphicsExtensions{.eps,.pdf,.png,.jpg}
\else
  \DeclareGraphicsExtensions{.eps}
\fi

\newsiamremark{remark}{Remark}
\newsiamremark{hypothesis}{Hypothesis}
\crefname{hypothesis}{Hypothesis}{Hypotheses}
\newsiamthm{claim}{Claim}

\newtheorem{assumption}{Assumption}

\newtheorem{model}{Model}

\newcommand{\bx}{\boldsymbol{x}}

\headers{Two-level overlapping Schwarz methods }{Q. Lu, J. Wang, S. Shu and J. Peng}

\title{Two-level overlapping Schwarz methods based on local generalized eigenproblems for Hermitian variational problems
\thanks{Qing Lu and Junxian Wang contributed equally to this work and should be considered co-first authors. Shi Shu is the corresponding author.}
\funding{Qing Lu is supported by Hunan Provincial Innovation Foundation for Postgraduate (CX20190465). Shi Shu is supported by National Natural Science Foundation of China (Grant No. 11971414).
}}

\author{Qing Lu\thanks{School of Mathematics and Computational Science, Hunan Key Laboratory for Computation and Simulation in Science and Engineering,
Xiangtan University, Xiangtan, Hunan 411105, China(\email{lu\_sunny93@163.com}, \email{wangjunxian@xtu.edu.cn}, \email{shushi@xtu.edu.cn}).}%
\and Junxian Wang\footnotemark[2]
\and Shi Shu\footnotemark[2]
\and
Jie Peng \thanks{School of Mathematical Sciences, South China Normal University, Guangzhou 510631, China(\email{pengjie18@m.scnu.edu.cn}).}}%

\usepackage{amsopn}

\begin{document}

\maketitle

\begin{abstract}
The research of two-level overlapping Schwarz (TL-OS) method based on constrained energy minimizing coarse space is still in its infancy, and there exist some defects, e.g. mainly for second order elliptic problem and too heavy computational cost of coarse space construction. In this paper, by introducing appropriate assumptions, we propose more concise coarse basis functions for general Hermitian positive and definite discrete systems, and establish the algorithmic and theoretical frameworks of the corresponding TL-OS methods. Furthermore, to enhance the practicability of the algorithm, we design two economical TL-OS preconditioners and prove the condition number estimate. 
As the first application of the frameworks,
we prove that the assumptions hold for the linear finite element discretization of second order elliptic  problem with high contrast and oscillatory coefficient and the condition number of the TL-OS preconditioned system is robust with respect to the model and mesh parameters. In particular, we also prove that the condition number of the economically preconditioned system is independent of the jump range under a certain jump distribution.
Experimental results show that the first kind of economical preconditioner is more efficient and stable than the existed one. Secondly, we construct TL-OS and the economical TL-OS preconditioners for the plane wave least squares discrete system of Helmholtz equation by using the frameworks. The numerical results for homogeneous and non-homogeneous cases illustrate that the PCG method based on the proposed preconditioners have good stability in terms of the angular frequency, mesh parameters and the number of degrees of freedom in each element.
\end{abstract}



\section{Introduction}\label{introduction}
Hermitian positive and definite discrete variational problems are widely appeared in the field of science and engineering computation. Domain decomposition method (DDM, \cite{TW2005}) is a popular method for solving large-scale discrete systems because of its natural parallelism and suitability for complex problems. DDM based on constrained energy-minimizing coarse space is an important method developed in recent years which can also be divided into nonoverlapping and overlapping types.
Balancing domain decomposition by constraints (BDDC) method which was firstly proposed by C. Dohrmann in 2003(\cite{D2003}) is a representative method(\cite{D2003}-\cite{PW2010}) of the former, and it can be treated as the dual form of dual-primal finite element tearing and interconnecting
(FETI-DP) method(\cite{LW2006}). In order to improve the robustness and universality of BDDC method for complex problems, adaptive BDDC emerges as the times require. It was firstly developed by J. Mandel and B. sousedik in 2007 (\cite{MS2007}) and used in various problems soon afterwards. Among them, adaptive BDDC method based on different generalized eigenproblems are designed for the finite element method (FEM) of second order elliptic problem in \cite{MSS2012}\cite{KC2015}\cite{KRR2016}. In \cite{DP2013}, parallel sum(\cite{AD1969}) was firstly introduced to construct generalized eigenproblem in adaptive BDDC method, and used in various discrete systems and various model problems for constructing appropriate generalized eigenproblems such as FEM and staggered discontinuous Galerkin systems of second order elliptic problem(\cite{KCW2017}\cite{KCX2017}), Raviart-Thomas element system(\cite{DSO2015}), saddle point problems arising from mixed formulations of Darcy flow in porous media(\cite{ZT2017}), the plane wave least squares (PWLS) discretization of Helmholtz equations(\cite{PWS2018}\cite{PSWZ2021}). Furthermore, a unified framework of adaptive BDDC method was proposed in \cite{PD2017}. Latterly adaptive BDDC is also extended to the asymmetric positive definite convection diffusion problem in \cite{PSWZ2021-2}.
While the overlapping DDM based on constrained energy-minimizing coarse space is still in its infancy so far. In \cite{JRG2010}, a coarse space with minimal constrained energy and the corresponding two level overlapping Schwarz(TL-OS) preconditioners are constructed for linear element equation of second order elliptic problem, but there exist some defects, e.g. the condition number is dependent on the overlap width. In \cite{GCEL}, a hybrid overlapping DDM preconditioner is developed for the linear element discretization of Darcy flow equation, \cite{WCK2020} proposed two kinds of TL-OS preconditioners based on the multiscale function introduced in \cite{CEL2018} for finite element discretization of second order elliptic problem, they are both aimed at second order elliptic problem and there are too much cost in the coarse space construction.
Therefore, how to design TL-OS preconditioners with lower computational complexity for more general problems needs further research.

In this paper, by introducing an auxiliary functional which satisfies a certain assumption, we construct more concise coarse basis functions based on constrained energy minimization for the discretizations of general Hermitian positive and definite variational problems, and establish the algorithmic and theoretical frameworks of the corresponding TL-OS preconditioned system. 
It is strictly proved that the condition number of the designed TL-OS preconditioned system only depends on the user-defined threshold $\Lambda$ and the maximum number of adjacent subregions $M$. However it still costs too much since the coarse basis functions should be solved in the whole region $\Omega$. In order to overcome this drawback, we design two economical TL-OS preconditioners $B^{-1}_{k,\psi}$ and $B^{-1}_{k,\bar{\psi}}$. Based on the above assumptions and a newly introduced assumption, we prove that the condition number is still bounded by a constant only dependent on $\Lambda$ and $M$ when $k$ is sufficiently large. It is worth pointing out that the coarse basis formula corresponding to preconditioner $B^{-1}_{k,\psi}$ is more concise, which greatly reduces the cost of coarse space construction in the existing work. 

As the first application of the above frameworks, a TL-OS preconditioner and two economical TL-OS preconditioners with lower computational complexity are designed for the finite element discretization of second order elliptic problem. It is proved strictly that all the three assumptions hold. 
In addition, we also prove that the condition number of the designed economical TL-OS preconditioned system is independent of the jump range for any given $k$ when the coefficient $\rho(\bx)$ satisfies a certain jump distribution. It is worth mentioning that the coefficient matrix of the first kind of coarse basis is sparse, so it can be solved efficiently with fast solver such as AMG. Various experimental results show that the first economical preconditioner is more stable than the second one. Next, by using the frameworks and introducing an auxiliary functional satisfying the assumption and a partition of unity defined on the dual partition, we construct a TL-OS and the corresponding economical TL-OS preconditioners for the PWLS discrete system of Helmholtz equation.
Numerical experiments are carried out for homogeneous and non-homogeneous cases which illustrate that the corresponding PCG method with TL-OS preconditioners has good stability, and PCG method with economical TL-OS preconditioners weakly depends on the model and mesh parameters when $k\ge 2$.

This paper is organized as follows. In section \ref{frameworks}, algorithmic and theoretical frameworks of a TL-OS and two economical TL-OS preconditioners with low computational complexity are established for general Hermitian positive and definite system. In section \ref{Poisson} and section \ref{Helmholtz}, based on the above frameworks, the TL-OS preconditioners and economical TL-OS preconditioners are designed for the linear finite element system of second order elliptic equation and the PWLS system of Helmholtz equation, respectively, and the numerical experiments are included.

\section{Hermitian variational problem and TL-OS preconditioned theory}\label{frameworks}
Let $V(\Omega)$ be a Hilbert space on complex filed and we consider the following continuous variational model problem: find $u\in V(\Omega)$, such that
\begin{eqnarray}\label{ucbianfen-continuous}
a(u,v)=f(v),~~\forall v\in V(\Omega),
\end{eqnarray}
where $\Omega\subset \mathbb{R}^d$($d\ge 1$) is a bounded polygonal (polyhedral) domain, $a(\cdot,\cdot)$ is a sesquilinear and Hermitian positive definite functional and $f(\cdot)$ is a bounded source functional. 

Let $V_h:=V_h(\Omega)$ be a discrete space associated with the partition $\mathcal{T}_h$ of $\Omega$ where $h$ is the size of the elements.
Then the approximation of model \eqref{ucbianfen-continuous} can be expressed as: find $u_h\in V_h$ such that
\begin{eqnarray}\label{ucbianfen}
a(u_h,v_h)=f(v_h),~~\forall v_h\in V_h,
\end{eqnarray}
and Hermitian positive definite operator $A_h: V_h\rightarrow V_h$ satisfies
\begin{eqnarray}\label{sys}
A_h u_h=f_h,~~~~u_h\in V_h.
\end{eqnarray}

In the following, we will discuss a two-level overlapping Schwarz (TL-OS) preconditioner for preconditioned conjugate gradient (PCG) method.

\subsection{TL-OS preconditioner}\label{framework-TL-OS}
We first introduce the nonoverlapping subdomain division $\bar{\Omega}=\bigcup_{i=1}^N \bar{\Omega}_i$ where $\Omega_i\bigcap\Omega_j=\emptyset(i\ne j)$ and $N$ is the number of subdomains.
Then each subdomain is extended by several layers of elements in $\mathcal{T}_h$ to obtain an overlapping subdomain partition $\{\Omega'_i\}^N_{i=1}$(\cite{WCK2020}). Let $2\delta$ be the minimum overlapping width between two overlapping subdomains, and $H$ be the size of $\{\Omega_i\}^N_{i=1}$.

Later in the paper, the restriction of sesquilinear form $c(\cdot,\cdot)$ on region $\mathcal{D}$ is usually abbreviated as $c(\cdot, \cdot)_{\mathcal{D}}$ where $\mathcal{D}\subset \bar{\Omega}$ may be a closed area. Define the norm
\begin{eqnarray*}\label{norm-abc}
|u|_{a(\mathcal{D})}^2 &=& a( u, u)_{\mathcal{D}},~~\| u\|_{c(\mathcal{D})}^2 = c( u, u)_{\mathcal{D}},~~c\ne a,
\end{eqnarray*}
especially when $\mathcal{D}=\bar{\Omega}_{i}$, we denote
$a_i(\cdot, \cdot):=a(\cdot, \cdot)_{\bar{\Omega}_{i}}.$

For a given overlapping subdomain $\Omega'_{j}$, we introduce the following subspace of $V_h$
\begin{eqnarray*}
V_h(\Omega'_j) = \{u_j\in V_h: supp(u_j)\subseteq \Omega'_j\}
\end{eqnarray*}
and local operator $A_j: V_h(\Omega'_j)\rightarrow V_h(\Omega'_j)$ such that
\begin{eqnarray}\label{tildeA-def}
 (A_j u_j, v_j)=a(u_j, v_j)_{\Omega'_{j}},~~ \forall u_j,~v_j\in V_h(\Omega'_j).
\end{eqnarray}

Then we can construct the overlapping Schwarz preconditioner $$B^{-1}_s = \Sigma_{j=1}^N \Pi_j A_j^{-1}\Pi^*_j,$$
where the identical prolongation operator $\Pi_{j}: V_h(\Omega'_j) \mapsto V_h,~~j=1,\cdots,N,$
and $\Pi^*_j$ is the conjugate operator of $\Pi_j$.
Now we will construct a coarse space for \eqref{sys} based on local generalized eigenproblems(\cite{KCW2017}\cite{PD2017}\cite{CEL2018}\cite{WCK2020}) to get a more robust preconditioner. 
For this, we introduce a partition of unity $\{\theta_i\}_{i=1}^N$ which satisfy
\begin{equation}\label{theta-def}
\Sigma_{i=1}^N \theta_i(x)=1,~0\le \theta_i(x)\le 1,~x\in \bar{\Omega}~~and~~
supp(\theta_i)\subset \bar{\Omega}'_{i}.
\end{equation}

Denote
\begin{eqnarray}\label{S-i-M}
\mathcal{S}^{(i)}=\{1\le j\le N: \Omega'_{j}\cap \Omega'_{i}\ne \emptyset\},~~
M=\max_{1\le i\le N}|\mathcal{S}^{(i)}|.
\end{eqnarray}


\begin{assumption}\label{assump-1-strong}
For any $u\in V_h$ and $\{\theta_j\}_{j=1}^N$ defined by \eqref{theta-def}, there exists a positive constant $\tilde{C}$ satisfying 
\begin{eqnarray}\label{inequality-s-a-strong}
a_i(\theta_j u,\theta_j u)&\le& \tilde{C}(a_i(u, u)+s_i(u,u)),~~j\in \mathcal{S}^{(i)},~~i=1,\cdots,N,
\end{eqnarray}
where $s_i(\cdot,\cdot)$ is some sesquilinear and Hermitian positive definite functional on $V_h(\bar{\Omega}_i)$, and for any $\eta\in L^{\infty}(\bar{\Omega}_{i})$,
\begin{eqnarray}\label{inequality-Ih-s}
\|\eta u\|_{s_i}&\le& \|\eta\|_{L^{\infty}(\bar{\Omega}_{i})}\|u\|_{s_i},~~\|u\|_{s_i}=\sqrt{s_i(u,u)}.
\end{eqnarray}
\end{assumption}

From \eqref{inequality-s-a-strong} and \eqref{S-i-M}, we have 
\begin{eqnarray}\label{inequality-s-a}
\Sigma_{i=1}^Na(\theta_i u,\theta_i u)_{\Omega'_{i}}&\le& \tilde{C}M\Sigma_{i=1}^N (a_i(u, u)+s_i(u,u)),~~\forall u\in V_h.
\end{eqnarray}
%

Let $V_h(\bar{\Omega}_i)=V_h|_{\bar{\Omega}_i}$, $n_i=dim(V_h(\bar{\Omega}_i))$, we introduce the local generalized eigenproblem on $\Omega_i$: find $\lambda_j^{(i)} \in \mathbb{C}$ and $\tilde{\phi}_j^{(i)} \in V_h(\bar{\Omega}_i)$ such that
\begin{equation}\label{GEP-general}
s_i(\tilde{\phi}_j^{(i)},w) = \lambda_j^{(i)} a_i(\tilde{\phi}_j^{(i)},w), ~~\forall w \in V_h(\bar{\Omega}_i),~j=1,\cdots,n_i.
\end{equation}
It is obvious that $\{\lambda_j^{(i)}\}_{j=1}^{n_i}$ are all positive real numbers. We assume that they are arranged in
descending order
\begin{eqnarray}\label{lambda-order1}
\lambda^{(i)}_{1} \geq  \cdots  \lambda^{(i)}_{{l_{i}}} \geq \Lambda  > \lambda^{(i)}_{l_{i+1}}  \geq  \cdots \geq \lambda^{(i)}_{n_i}>0,
\end{eqnarray}
where $\Lambda>1$ is a user defined threshold and $l_{i}$ is a nonnegative integer.


\begin{lemma}\label{deltaij-a-s}
There exists a series of eigenfunctions $\{\phi_j^{(i)}\}_{j=1}^{n_i}$ corresponding to $\{\lambda_j^{(i)}\}_{j=1}^{n_i}$ such that
\begin{eqnarray}\label{delta-ij}
\left\{
\begin{array}{lll}
&s_i(\phi_j^{(i)}, \phi_l^{(i)})=0=a_i(\phi_j^{(i)}, \phi_l^{(i)}),~~j\ne l,\\
&s_i(\phi_j^{(i)}, \phi_j^{(i)})=1.
\end{array}
\right.
\end{eqnarray}
\end{lemma}

For ease of notation, we denote the eigenfunctions extended by zero on $\Omega$ to be still $\{\phi_j^{(i)}\}_{j=1}^{n_i}$, 
and introduce some local auxiliary spaces(\cite{CEL2018})
\begin{equation}\label{vauxi-def}
V_{aux}^{(i)}: = span\{\phi_1^{(i)},\cdots,\phi_{l_i}^{(i)}\},~~i=1,\cdots,N
\end{equation}
and a global auxiliary space
$V_{aux}: = \oplus_{i=1}^N V_{aux}^{(i)}.$

Using the above auxiliary spaces, we can define the coarse space $V_0$ as follows
\begin{align}\label{zjb-def}
V_h = \tilde{V}\perp_{a(\cdot,\cdot)} V_0,
\end{align}
where $\perp_{a(\cdot,\cdot)}$ denotes the orthogonality under the inner product $a(\cdot,\cdot)$,
\begin{equation}\label{tilde-V}
\tilde{V}: = \{v\in V_h|s_i(v,w)=0, ~~\forall w\in V_{aux}^{(i)},~~i=1,\cdots,N\}.
\end{equation}
It is easy to know that
$dim(\tilde{V}) = n_h-\sum_{i=1}^N l_i$, where $n_h=dim(V_h)$.

We then introduce functions: find $\psi_j^{(i)}\in V_h$ such that
\begin{eqnarray}\label{psi-sub2-new}
a(\psi_j^{(i)},v) =s_i(\phi_j^{(i)},\pi_iv),~\forall v\in V_h,~~j=1,\cdots,l_i,~~i=1,\cdots,N,
\end{eqnarray}
where the projection operator $\pi_i: V_h \rightarrow V_{aux}^{(i)}$ 
satisfies that for a given function $v\in V_h$
\begin{eqnarray}\label{pi-defsource}
s_i(\pi_i v,w) = s_i(v,w),~~\forall w\in V_{aux}^{(i)}.
\end{eqnarray}

It is obvious that functions $\{\psi_j^{(i)}\}$ defined by \eqref{psi-sub2-new} can span the coarse space $V_0$.
%
For real case, another basis $\{\bar{\psi}_j^{(i)}\}$ of $V_0$ are proposed in \cite{WCK2020} which satisfy
\begin{eqnarray}\label{psi-sub2}
b(\bar{\psi}_j^{(i)},v)=s_i(\phi_j^{(i)},\pi_iv),~\forall v\in V_h,~~j=1,\cdots,l_i,~i=1,\cdots,N,
\end{eqnarray}
where $\pi_i$ is defined by \eqref{pi-defsource},
\begin{eqnarray}\label{b}
b(u,v)=a(u,v) + \Sigma_{l=1}^Ns_l(\pi_l u,\pi_lv),~~\forall u,v \in V_h.
\end{eqnarray}

Comparing the variational problem \eqref{psi-sub2-new} and \eqref{psi-sub2}, it is easy to find that the corresponding coefficient matrix of the former is usually better than that of the latter. 
so the computational complexity of the new basis $\{\psi_j^{(i)}\}$ is lower than $\{\bar{\psi}_j^{(i)}\}$.

We define a restriction of operator $A_h$ on $V_0$ as $A_0: V_0\rightarrow V_0$ which satisfies
\begin{eqnarray}\label{tildeA0-def}
 (A_0u_0, v_0)=a(u_0, v_0),~~ \forall u_0, v_0 \in V_0,
\end{eqnarray}
and introduce an identical lifting operator $\Pi_0: V_0 \mapsto V_h$.

Using the operators $A_0$ and $\Pi_0$, we can construct a TL-OS preconditioner for \eqref{sys} based on local generalized eigenproblems as follows
\begin{eqnarray}\label{B-def-new}
  B^{-1} = \Sigma_{j=1}^N \Pi_j A_j^{-1}\Pi^*_j+\Pi_0 A_0^{-1}\Pi^*_0.
\end{eqnarray}


\subsection{Condition number estimate of $B^{-1}A_h$}\label{framework-TL-OS-theory}
Fistly, noting that $a(\cdot,\cdot)$ is sesquilinear and the support of $u_i(i=1,\cdots,N)$ is $\Omega'_i$, we can prove 
\begin{lemma}\label{theorem-1}
For any $u_0\in V_0, u_i \in V_h(\Omega'_i)(i=1,\cdots,N)$, we have
\begin{eqnarray}\label{max-proof2}
|\Sigma_{i=1}^{N}\Pi_i
 u_i+\Pi_0
 u_0|_{a(\Omega)}^2 \le 2M(\Sigma_{i=1}^{N}| u_i|_{a(\Omega'_i)}^2+| u_0|_{a(\Omega)}^2),
\end{eqnarray}
where the constant $M$ is defined by \eqref{S-i-M}.
\end{lemma}
%

For a given $u\in V_h$, define $u_0\in V_0$ as
\begin{eqnarray}\label{u0-def-1}
a(u_0,v) = a(u,v),~\forall v\in V_0, ~~namely,~~ u-u_0\in V_0^\perp=\tilde{V}.
\end{eqnarray}

We introduce an interpolation operator: $I_h: L^2(\Omega)\rightarrow V_h$ which satisfies the following stability assumptions.

\begin{assumption}\label{assump-2}
For any $u\in V(\Omega)$, there exist constants $C_{I_a}$ and $C_{I_s}$ satisfying
\begin{eqnarray}\label{Ih-bound}
a(I_h u, I_h u)_{\tau}&\le&  C_{I_a}a(u, u)_{\tau},~~\forall \tau\in \mathcal{T}_h,\\
\Sigma_{l=1}^N \|I_h u\|_{s_l}^2
&\le& C_{I_s} \Sigma_{l=1}^N \|u\|_{s_l}^2.\label{Ih-bound-s}
\end{eqnarray}
\end{assumption}

Set $C_I = \max\{C_{I_a},C_{I_s}\}$.
For $\theta_i$ defined by \eqref{theta-def}, let
\begin{eqnarray}\label{ui-def}
u_i=I_h(\theta_{i}(u-u_0)) \in V_h(\Omega'_i),~ i=1,\cdots, N.
\end{eqnarray}

Using the definition of $I_h$ and $\{\theta_i\}_{i=1}^N$, and noting that $\Pi_i(i=0,\cdots,N)$ are all identical operators, we have the following decomposition
\begin{eqnarray}\label{decomp-u}
u=\Sigma_{i=0}^N\Pi_iu_i=\Sigma_{i=0}^N u_i,
\end{eqnarray}
where $u_0$ and $u_i(i=1,\cdots,N)$ are defined by \eqref{u0-def-1} and \eqref{ui-def} respectively.

Noting that $u-u_0\in \tilde{V}$, $\{\phi_j^{(i)}\}_{j=1}^{n_i}$ are the basis functions of $V_h(\bar{\Omega}_i)$, using \eqref{u0-def-1}, \eqref{delta-ij} and \eqref{GEP-general}, we can easily prove the following two properties of $u-u_0$.
\begin{lemma}\label{lemma-1}
For a given $u\in V_h$ and $u_0$ is defined by \eqref{u0-def-1}, we have
\begin{eqnarray}\label{w-omegai}
(u-u_0)|_{\bar{\Omega}_i} =\Sigma_{j=l_i+1}^{n_i}C_{\Delta,j} \phi_j^{(i)},
\end{eqnarray}
where $C_{\Delta,j} \in \mathbb{C}(j=l_i+1,\cdots,n_i)$ are all constants.
\end{lemma}

%
%

\begin{lemma}\label{lemma-2}
For a given $u\in V_h$ and $u_0\in V_0$ is defined by \eqref{u0-def-1}, we have
\begin{equation}\label{CondA-min1-real}
   s_i(u-u_0,u-u_0) \leq \Lambda a_i(u-u_0,u-u_0).
 \end{equation}
\end{lemma}

With the preparations above, the following stability lemma about decomposition \eqref{decomp-u} comes naturally.

\begin{lemma}\label{theorem-nonecnomic}
Assume that the ASSUMPTION \ref{assump-1-strong} and \ref{assump-2} both hold, then for a given $u\in V_h$, the decomposition defined by \eqref{decomp-u} is stable, i.e.
\begin{eqnarray}\label{CondA2-z}
\Sigma_{i=0}^N|u_i|_{a(\Omega'_i)}^2 \le C_{I_a}\tilde{C}M(1+\Lambda) |u|^2_{a(\Omega)},
\end{eqnarray}
where $C_{I_a}$ and $\tilde{C}$ are determined by \eqref{Ih-bound} and \eqref{inequality-s-a-strong}, respectively.
\end{lemma}

\begin{proof}
It suffices to prove that
\begin{eqnarray}\label{lambda-min-proof}
\Sigma_{i=1}^N a(u_i,u_i)_{\Omega'_{i}}+a(u_0,u_0)\le C_{I_a}\tilde{C}M(1+\Lambda) a(u,u).
\end{eqnarray}

In fact, using \eqref{ui-def}, \eqref{Ih-bound}, \eqref{inequality-s-a}, \eqref{CondA-min1-real} and \eqref{u0-def-1}, we can obtain
\begin{eqnarray*}\nonumber
    \Sigma_{i=1}^N a(u_i,u_i)_{\Omega'_{i}}
    &\le& C_{I_a}\Sigma_{i=1}^Na(\theta_i(u-u_0),\theta_i(u-u_0))_{\Omega'_{i}} \\\nonumber
    &\le& C_{I_a}\tilde{C}M(\Sigma_{i=1}^Na_i(u-u_0, u-u_0)+\Sigma_{i=1}^Ns_i(u-u_0,u-u_0)) \\\label{proof-min-1}
    &\le& C_{I_a}\tilde{C}M(\Sigma_{i=1}^Na_i(u-u_0,u-u_0)+\Lambda \Sigma_{i=1}^Na_i(u-u_0,u-u_0)) \\\nonumber
    &\le&C_{I_a}\tilde{C}M(1+\Lambda) a(u,u),
\end{eqnarray*}
from this and noting $a(u_0,u_0)\le a(u,u)$, which finish the proof of \eqref{lambda-min-proof}.
\end{proof}

Combining LEMMA \ref{theorem-1}, LEMMA \ref{theorem-nonecnomic} and the theorem 2.2 in \cite{HX2007}, we can obtain
\begin{theorem}\label{theorem-general}
If the ASSUMPTION \ref{assump-1-strong} and \ref{assump-2} both hold, for the preconditioner $B^{-1}$ defined by \eqref{B-def-new}, we have
\begin{eqnarray}\label{cond-general}
\kappa(B^{-1}A_h)\le C,
\end{eqnarray}
where the positive constant $C$ depends only on $\Lambda$ and $M$.
\end{theorem}

We note that the TL-OS precontioner defined by \eqref{B-def-new} needs high computational cost since the coarse basis determined by \eqref{psi-sub2-new} (or \eqref{psi-sub2}) require the solution on entire $\Omega$. In next subsection, two economical preconditioners with lower computational cost and better parallel scalability will be discussed.
It is worth pointing out that the variational form \eqref{psi-sub2-new} will be used to calculate economical coarse basis because of their lower computational complexity.


\subsection{Two economical TL-OS preconditioners}\label{framework-economical}

For a given nonoverlapping subdomain $\Omega_i$, we introduce a region $\Omega_{k,H}^{(i)}\subset \Omega$ by extending $\Omega_i$ with $k$($k$ is an integer) layers of neighbouring subdomains, see Fig. \ref{coarse-space-grid}, the green part shows some $\Omega_{k,H}^{(i)}$ with $k=1$ in two dimension.




\begin{figure}[h]
\centering
\subfigure[]{\label{coarse-space-grid}
\includegraphics[width=1.0in]{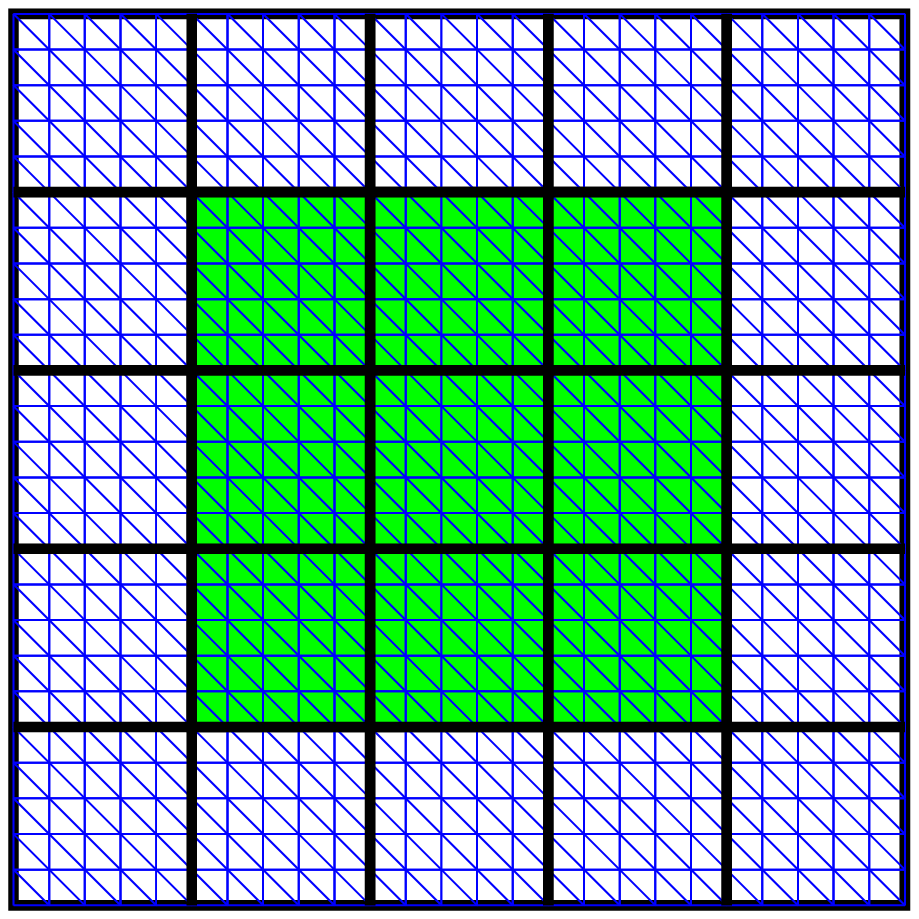}}
\hspace{0.01in}
\subfigure[]{\label{subdomain-interior}
\includegraphics[width=1.0in]{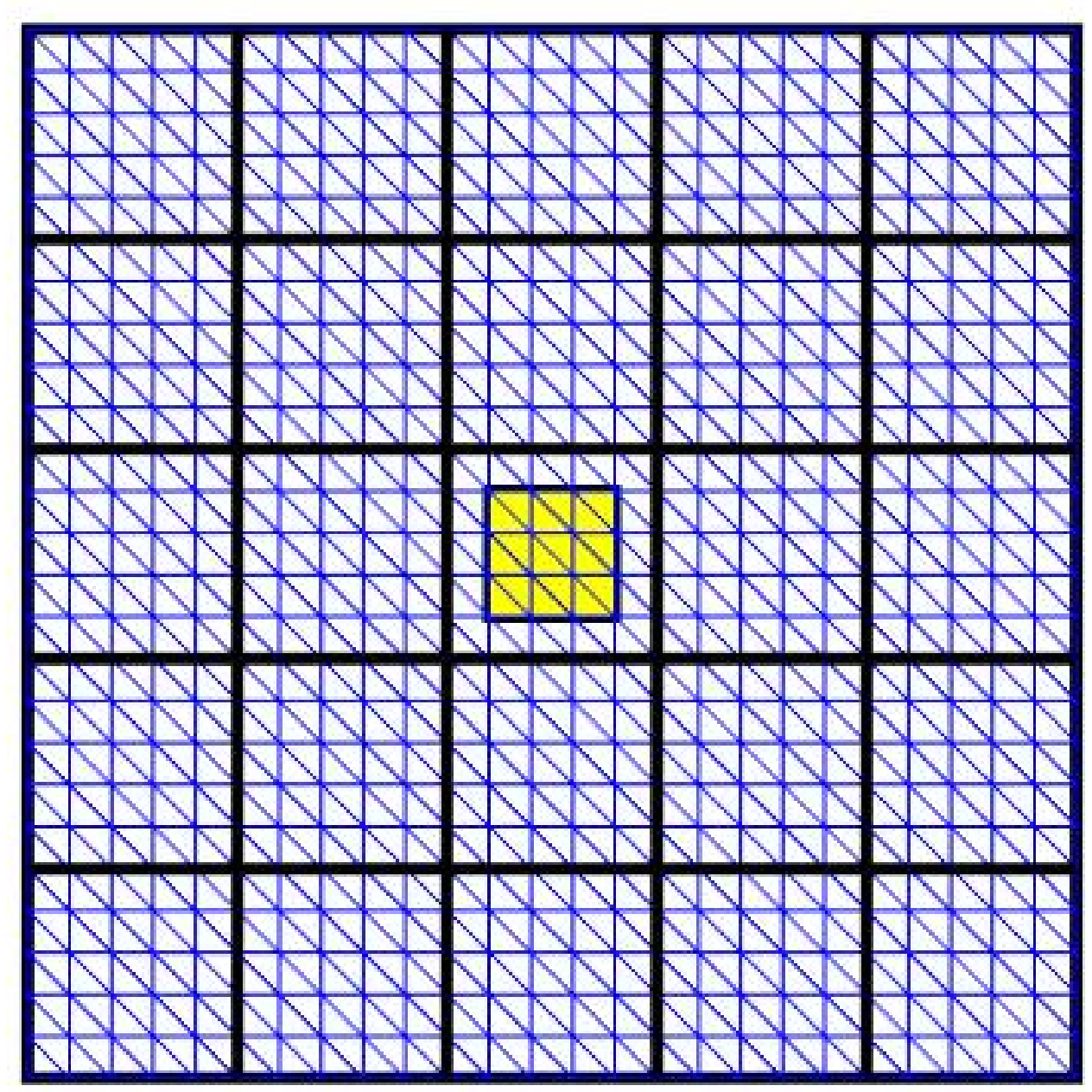}}
\hspace{0.01in}
\subfigure[]{ \label{1d-theati}
\includegraphics[width=1.0in]{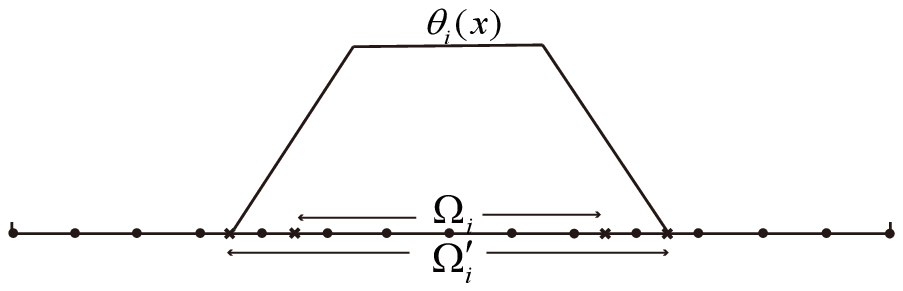}}
\hspace{0.01in}
{\caption{(a)~The green domain is $\Omega_{k,H}^{(i)}$ with $k=1$, (b)~The yellow domain is $\Omega_{i,INT}$,(c) ~$\theta_i(x)$}
\label{Fig:submesh-2-new}}
\end{figure}


%

Define the following finite element space on $\Omega_{k,H}^{(i)}$
$$V_{k,H}^{(i)}=\{v\in V_h | supp(v)\subseteq \Omega_{k,H}^{(i)}\}.$$

Based on this space and similar to \eqref{psi-sub2-new} and \eqref{psi-sub2}, we can give the approximations of $\psi_j^{(i)}$ and $\bar{\psi}_j^{(i)}$ on $\Omega_{k,H}^{(i)}$: find $\zeta_{j,k}^{(i)} \in  V_{k,H}^{(i)}$($\zeta=\psi, \bar{\psi}$) such that
\begin{eqnarray}\label{psi-ec-I}
a(\psi_{j,k}^{(i)},v) &=&s_i(\phi_j^{(i)},\pi_iv),~~\forall v\in V_{k,H}^{(i)},\\
b(\bar{\psi}_{j,k}^{(i)},v) &=& s_i(\phi_j^{(i)},\pi_iv),~~\forall v\in V_{k,H}^{(i)},\label{psi-ec-II}
\end{eqnarray}
where $b(\cdot,\cdot)$ is defined by \eqref{b}, $s_i(\cdot,\cdot)$ is determined by the inequality \eqref{inequality-s-a-strong}, $\phi_j^{(i)}$ is the basis function of \eqref{vauxi-def}.

Then we can define the following two economical coarse spaces
\begin{eqnarray*}
V^{k,\zeta}_{0}=span\{\zeta_{j,k}^{(i)},~~j=1,\cdots,l_i,~~i=1,\cdots,N\},~~\zeta=\psi, \bar{\psi}.
\end{eqnarray*}

Similar to \eqref{B-def-new}, we obtain two economical TL-OS preconditioners
\begin{eqnarray*}
B^{-1}_{k,\zeta}= \Sigma_{j=1}^N \Pi_j A_j^{-1}\Pi^*_j+\Pi^{k,\zeta}_{0} (A^{k,\zeta}_{0})^{-1}(\Pi^{k,\zeta}_{0})^*,~~\zeta=\psi, \bar{\psi},
\end{eqnarray*}
where $\Pi^{k,\zeta}_{0}: V^{k,\zeta}_{0}\rightarrow V_h$ is an identical operator and coarse space operator $A^{k,\zeta}_{0}: V^{k,\zeta}_{0}\rightarrow V^{k,\zeta}_{0}$ satisfies
$(A^{k,\zeta}_{0}u,v)=a(u,v),~~\forall u,v\in V^{k,\zeta}_{0}$.

Next we will discuss the estimate of condition number $\kappa(B^{-1}_{k,\zeta} A_h)$. Firstly, the following lemma can be obtained similar to LEMMA \ref{theorem-1}.
\begin{lemma}\label{theorem-ecnomic-1}
For any $u^{k,\zeta}_{0}\in V^{k,\zeta}_{0}, u_i \in V_h(\Omega'_i)(i=1,\cdots,N)$, we have
\begin{eqnarray}\label{max-ecnomic-proof2}
|\Sigma_{i=1}^{N}\Pi_iu_i+\Pi^{k,\zeta}_{0} u^{k,\zeta}_{0}|_{a(\Omega)}^2 \le 2M
(\Sigma_{i=1}^{N}| u_i|_{a(\Omega'_i)}^2+| u^{k,\zeta}_{0}|_{a(\Omega)}^2).
\end{eqnarray}
\end{lemma}

Noting that both $\{\psi_j^{(i)}\}$ and $\{\bar{\psi}_j^{(i)}\}$ are basis of $V_0$, therefore the expressions of $u_0\in V_0$ defined by \eqref{u0-def-1} under these two basis can be expressed as
\begin{eqnarray}\label{u0-zeta}
u^{\zeta}_{0} = \Sigma_{i=1}^N\Sigma_{j=1}^{l_i} c_{ij}\zeta_j^{(i)}.
\end{eqnarray}

Using the above coefficients $c_{ij}$, we can define the function
\begin{align}\label{chuji2-1}
u^{k,\zeta}_{0}= \Sigma_{i=1}^N\Sigma_{j=1}^{l_i} c_{ij}\zeta_{j,k}^{(i)}.
\end{align}

For $\{\theta_i\}_{i=1}^N$ defined by \eqref{theta-def}, let
\begin{eqnarray}\label{ui-def-ecnomic}
u^{\zeta}_{i}=I_h(\theta_{i}(u-u^{k,\zeta}_{0})) \in V_h(\Omega'_i),~ i=1,\cdots, N.
\end{eqnarray}
Similar to \eqref{decomp-u}, there exist two kinds of decomposition about a given $u\in V_h$
\begin{eqnarray*}
u=\Sigma_{i=1}^N\Pi_iu^{\zeta}_{i}+\Pi^{k,\zeta}_{0}u^{k,\zeta}_{0}=\Sigma_{i=1}^N u^{\zeta}_{i}+u^{k,\zeta}_{0},~~\zeta=\psi, \bar{\psi}.
\end{eqnarray*}

Using \eqref{inequality-s-a}, \eqref{Ih-bound}, \eqref{CondA-min1-real}, and similar to (5.3) and the part below (5.3) in \cite{WCK2020}, we can prove that $u^{k,\zeta}_{0}$ and $u^{\zeta}_{i}$ separately defined by \eqref{chuji2-1} and \eqref{ui-def-ecnomic} meet
\begin{eqnarray}\label{xingzhi1}
a(u^{k,\zeta}_{0},u^{k,\zeta}_{0}) &\le& 6a(u,u) + 4a(u^{\zeta}_{0}-u^{k,\zeta}_{0},u^{\zeta}_{0}-u^{k,\zeta}_{0}),\\
\Sigma_{i=1}^Na(u^{\zeta}_{i},u^{\zeta}_{i})&\le& 2C_{I_a}\tilde{C}M((1+\Lambda)a(u,u)+ b(u^{\zeta}_{0}-u^{k,\zeta}_{0},u^{\zeta}_{0}-u^{k,\zeta}_{0})).  \label{xingzhi2}
\end{eqnarray}
Then we obtain
\begin{eqnarray}
&&\label{CondA2}\\
&&\Sigma_{i=1}^N|u^{\zeta}_{i}|_{a(\Omega'_i)}^2 + |u^{k,\zeta}_{0}|_{a(\Omega)}^2\le  (6+2C_{I_a}\tilde{C}M(1+\Lambda))|u|_{a(\Omega)}^2+ (4+2C_{I_a}\tilde{C}M)\|w\|_{b(\Omega)}^2,\nonumber
\end{eqnarray}
where
\begin{eqnarray}\label{norm-b}
\|w\|_{b(\Omega)}^2 = b(w, w), ~~w=u^{\zeta}_{0}-u^{k,\zeta}_{0}.
\end{eqnarray}

Denote the functions column vectors
\begin{eqnarray*}
{\bf \zeta}^{(i)}=(\zeta_1^{(i)},\cdots,\zeta_{l_i}^{(i)})^T,~~
{\bf \zeta}^{(i)}_k=(\zeta_{1,k}^{(i)},\cdots,\zeta_{l_i,k}^{(i)})^T,~~
\Phi_{\pi}^{(i)}=(\phi_1^{(i)},\cdots,\phi_{l_i}^{(i)})^T,
\end{eqnarray*}
and for a given $l_i$-dimensional vector $\vec{c}=(c_{i1},\cdots,c_{il_i})\in \mathbb{C}^{l_i}$, we introduce functions
\begin{eqnarray}\label{lemma-0}
w_{\nu} = \vec{c}\nu,~~ \nu = {\bf \zeta}^{(i)},{\bf \zeta}^{(i)}_k,\Phi_{\pi}^{(i)}.
\end{eqnarray}

From \eqref{norm-b},\eqref{u0-zeta}, \eqref{chuji2-1} and \eqref{lemma-0}, we know that
\begin{eqnarray}\label{lemma4-1-1}
w &=&  \Sigma_{i=1}^N \Sigma_{j=1}^{l_i}c_{ij}(\zeta_j^{(i)}-\zeta_{j,k}^{(i)}) := \Sigma_{i=1}^N w_i
\end{eqnarray}
where
\begin{eqnarray}\label{lemma4-1-2}
w_i=\Sigma_{j=1}^{l_i}c_{ij}(\zeta_j^{(i)}-\zeta_{j,k}^{(i)})=w_{{\bf \zeta}^{(i)}}-w_{{\bf \zeta}_k^{(i)}}.
\end{eqnarray}

In order to estimate the right hand side $\|w\|_{b(\Omega)}^2$ of \eqref{CondA2}, functions
\begin{eqnarray}\label{chi-i-k}
\chi_i^k = \Sigma_{\Omega_l\subset \Omega^{(i)}_{k,H}} \theta_l,~~i=1,\cdots,N
\end{eqnarray}
are introduced which associated with $\Omega^{(i)}_{k,H}$, it is easy to know from \eqref{theta-def} that
\begin{eqnarray}\label{chik-1}
1-\chi_i^{k}=0,~~on ~\Omega^{(i)}_{k-1,H}, ~~and ~~1-\chi_i^{k}=0,~~on ~\Omega\backslash\Omega^{(i)}_{k+1,H}.
\end{eqnarray}

For any $u\in V_h$, from \eqref{theta-def} and \eqref{inequality-s-a-strong}, we can derive
\begin{eqnarray}
|\sum\limits_{\Omega_l\subset \mathcal{O}} \theta_lu|_{a(\mathcal{D})}^2
&=&\sum\limits_{\Omega_j\subset \mathcal{D}}|\sum\limits_{\Omega_l\subset \mathcal{O}} \theta_lu|_{a(\Omega_j)}^2
=\sum\limits_{\Omega_j\subset \mathcal{D}}|\sum\limits_{l\in S^{(j)}} \theta_lu|_{a(\Omega_j)}^2\nonumber\\
&\le&\sum\limits_{\Omega_j\subset \mathcal{D}}
\tilde{C}|S^{(j)}|^2(a_j(u, u)+s_j(u,u))\le\tilde{C}M^2\sum\limits_{\Omega_j\subset \mathcal{D}}
(a_j(u, u)+s_j(u,u)), \label{chi-i-k-estimate}
\end{eqnarray}
where the region $\mathcal{R}=\cup_{m\in \mathcal{S}_{\mathcal{R}}}\Omega_m$ and $\mathcal{S}_{\mathcal{R}}\subset \{1,\cdots,N\}$ ($\mathcal{R}=\mathcal{O}, \mathcal{D}$) is any index set.

By the definition \eqref{pi-defsource} of $\pi_l$, and using \eqref{GEP-general}, \eqref{lambda-order1} and \eqref{delta-ij}, we can obtain 
\begin{lemma}\label{lemma-pre}
For any $v\in V_h$, we have
    \begin{eqnarray}\label{lemma2-proof1}
      \|v\|_{s_l}^2\leq 2\Lambda|v|_{a(\Omega_l)}^2+3\|\pi_lv\|_{s_l}^2.
    \end{eqnarray}
\end{lemma}
%
%
%
%

In the following, we might as well take $\zeta=\bar{\psi}$ as an example to discuss the estimation of $\|w\|_{b(\Omega)}^2$.
It is similar to the proof of lemma 5.5 in \cite{WCK2020}, we can obtain 
\begin{lemma}\label{lemma5}
If both the ASSUMPTION \ref{assump-1-strong} and \ref{assump-2} hold, and $k\geq 2$, then for $w$ and $w_i$ defined by \eqref{lemma4-1-1} and \eqref{lemma4-1-2}, we have
\begin{eqnarray}\label{lemma5-0}
\|w\|_{b(\Omega)}^2 &\leq& \tilde{C}_1\Lambda(2k+1)^d \Sigma_{i=1}^N \|w_i\|_{b(\Omega)}^2,
\end{eqnarray}
\end{lemma}
where $\tilde{C}_1=3(\tilde{C}+1)C_I M^2$, $M$ and $\tilde{C}$ are separately defined by \eqref{S-i-M} and \eqref{inequality-s-a-strong}.
\begin{proof}
Using \eqref{lemma4-1-1}, \eqref{psi-sub2} and \eqref{chik-1}, we can easily obtain
\begin{eqnarray*}
b(w_i,I_h(\chi_i^{k-1}w))=0,~~~~b(w_i,I_h((1-\chi_i^{k+1})w))=0,
\end{eqnarray*}
from this and noting that $w\in V_h$, we have
\begin{eqnarray*}\nonumber
      \|w\|_{b(\Omega)}^2 &=& \Sigma_{i=1}^Nb(w_i,I_hw)=\Sigma_{i=1}^N b(w_i,I_h((1-\chi_i^{k+1}+\chi_i^{k+1}-\chi_i^{k-1}+\chi_i^{k-1})w))\\
                &=& \Sigma_{i=1}^N b(w_i,I_h((\chi_i^{k+1}-\chi_i^{k-1})w)).
\end{eqnarray*}

Utilizing \eqref{b} for the right side of the above equation, we know
\begin{eqnarray*}\nonumber
\|w\|_{b(\Omega)}^2
&\le&\Sigma_{i=1}^N\big(|a(w_i,I_h((\chi_i^{k+1}-\chi_i^{k-1})w))|
+\Sigma_{l=1}^N|s_l(\pi_lw_i,\pi_l[I_h((\chi_i^{k+1}-\chi_i^{k-1})w)])|\big),
\end{eqnarray*}
let $\Omega^{(i,H)}_{k,2}=\Omega^{(i)}_{k+2,H}\backslash \Omega^{(i)}_{k-2,H}$, and use \eqref{Ih-bound}, \eqref{Ih-bound-s}, the boundedness of $\pi_i$, \eqref{chik-1} and \eqref{inequality-Ih-s}, we derive
\begin{eqnarray}\nonumber
\|w\|_{b(\Omega)}^2
&\leq& C_I^{\frac{1}{2}}(\Sigma_{i=1}^N(|w_i|_{a(\Omega)}^2+\Sigma_{l=1}^N\|\pi_lw_i\|_{s_l}^2))^{\frac{1}{2}}
\cdot\\\label{w-normb-1}
&&(\Sigma_{i=1}^N(|(\chi_i^{k+1}-\chi_i^{k-1})w|_{a(\Omega^{(i,H)}_{k,2})}^2+\Sigma_{\Omega_l\subset \Omega^{(i,H)}_{k,2} }\|w\|_{s_l}^2))^{\frac{1}{2}}.
\end{eqnarray}

For the right side of \eqref{w-normb-1}, we use \eqref{chi-i-k}, \eqref{chi-i-k-estimate}, \eqref{lemma2-proof1}, and deduce
\begin{eqnarray*}\nonumber
\|w\|_{b(\Omega)}^2
&\leq& C_I^{\frac{1}{2}}(\Sigma_{i=1}^N(|w_i|_{a(\Omega)}^2+\Sigma_{l=1}^N\|\pi_lw_i\|_{s_l}^2))^{\frac{1}{2}}
\cdot\\\nonumber
&&(\Sigma_{i=1}^N(|\sum\limits_{\Omega_l\subset \Omega^{(i)}_{k+1,H}\backslash \Omega^{(i)}_{k-1,H} }\theta_l w|_{a(\Omega^{(i,H)}_{k,2})}^2+\sum\limits_{\Omega_l\subset \Omega^{(i,H)}_{k,2} }\|w\|_{s_l}^2))^{\frac{1}{2}}\\\nonumber
&\leq& C_I^{\frac{1}{2}}(1+2\Lambda)^{\frac{1}{2}}(\tilde{C}M^2+1)^{\frac{1}{2}}(\Sigma_{i=1}^N(|w_i|_{a(\Omega)}+\Sigma_{l=1}^N\|\pi_lw_i\|_{s_l}^2))^{\frac{1}{2}}\\\nonumber
&&(\Sigma_{i=1}^N(|w|_{a(\Omega^{(i,H)}_{k,2})}^2+\sum\limits_{\Omega_l\in \Omega^{(i,H)}_{k,2} }\|\pi_lw\|_{s_l}^2))^{\frac{1}{2}},
\end{eqnarray*}
from this and noting that each $\Omega_l$ is overlapped $k$ times on each direction, we get
\begin{eqnarray}\nonumber
\|w\|_{b(\Omega)}^2&\le& (2k+1)^{\frac{d}{2}}C_I^{\frac{1}{2}}(1+2\Lambda)^{\frac{1}{2}}(\tilde{C}M^2+1)^{\frac{1}{2}}
(\Sigma_{i=1}^N\|w_i\|_{b(\Omega)}^2)^{\frac{1}{2}}\|w\|_{b(\Omega)}
\end{eqnarray}
which completes the proof of \eqref{lemma5-0}.
\end{proof}

Inspired by \cite{WCK2020}, we obtain the following exponential decay
property for $w_i$ which is the key property of this method.
\begin{lemma}\label{lemma-key}
If both the ASSUMPTION \ref{assump-1-strong} and \ref{assump-2} hold, and $k=2m(m\geq1)$, we have
\begin{eqnarray}\label{lemma2-4-new}
 \|w_i\|_{b(\Omega)}^2 &\leq& \tilde{C}_2\Lambda E^{1-m} \|w_{\bar{\Psi}^{(i)}}\|_{b(\Omega)}^2,
\end{eqnarray}
where the constant $\tilde{C}_2=3\tilde{C}M^2C_I$, $E = ((\hat{C}(1+\Lambda))^{-1}+1) > 1$, $\hat{C}=(\frac{3}{2}\tilde{C}M^2+2)C_I$.
\end{lemma}
  \begin{proof}
By \eqref{psi-sub2}, \eqref{psi-ec-II}, and noting \eqref{lemma-0}, we have
\begin{eqnarray*}
b(w_{\bar{\Psi}^{(i)}}-w_{\bar{\Psi}_k^{(i)}},v) = 0,~~\forall v\in V_{k,H}^{(i)},
\end{eqnarray*}
from this and noting that $\bar{v}:=I_h(\chi_i^{k-1} w_{\bar{\Psi}^{(i)}}) \in V_{k,H}^{(i)}$, we can get
\begin{eqnarray}\label{lemma2-s2-proof1-new}
\|w_{i}\|_{b(\Omega)}^2 =\min_{\forall v\in V_{k,H}^{(i)}}\|w_{\bar{\Psi}^{(i)}}-v\|_{b(\Omega)}^2\le\|w_{\bar{\Psi}^{(i)}}-\bar{v}\|_{b(\Omega)}^2.
\end{eqnarray}

Use \eqref{chik-1}, \eqref{chi-i-k-estimate} and \eqref{lemma2-proof1}, let $\Omega^{(i,o)}_{k-2,H} =\Omega\backslash\Omega^{(i)}_{k-2,H}$, we have
\begin{eqnarray}\nonumber
|w_{\bar{\Psi}^{(i)}}-\bar{v}|_{a(\Omega)}^2
&=& C_{I_a} |(1-\chi_i^{k-1} )w_{\bar{\Psi}^{(i)}}|_{a(\Omega^{(i,o)}_{k-2,H})}^2\\
&\le& C_{I_a}\tilde{C}M^2(|w_{\bar{\Psi}^{(i)}}|_{a(\Omega^{(i,o)}_{k-2,H})}^2+\sum\limits_{\Omega_l\subset                {\Omega^{(i,o)}_{k-2,H}}}\|w_{\bar{\Psi}^{(i)}}\|_{s_l}^2) \nonumber\\\label{wpsi-barv-a}
&\leq& C_{I_a}\tilde{C}M^2(1+2\Lambda) \|w_{\bar{\Psi}^{(i)}}\|_{b(\Omega^{(i,o)}_{k-2,H})}^2.
\end{eqnarray}

By the boundedness of $\pi_l$, \eqref{Ih-bound-s}, \eqref{inequality-Ih-s} and \eqref{lemma2-proof1}, we deduce
\begin{eqnarray}\nonumber
\sum\limits_{l=1}^N \|\pi_l(w_{\bar{\Psi}^{(i)}}-\bar{v})\|_{s_l}^2
&\le& C_{I_s} \sum\limits_{\Omega_l\subset {\Omega^{(i,o)}_{k-2,H}}}\|(1-\chi_i^{k-1} ) w_{\bar{\Psi}^{(i)}}\|_{s_l}^2 \\\label{wpsi-barv-sl}
&\leq& C_{I_s} \sum\limits_{\Omega_l\subset {\Omega^{(i,o)}_{k-2,H}}}\|w_{\bar{\Psi}^{(i)}}\|_{s_l}^2 \leq C_{I_s}(1+2\Lambda)\|w_{\bar{\Psi}^{(i)}}\|_{b(\Omega^{(i,o)}_{k-2,H})}^2.
\end{eqnarray}

Combining \eqref{wpsi-barv-a}, \eqref{wpsi-barv-sl}, the definition of $\|\cdot\|_{b(\Omega)}$ and \eqref{lemma2-s2-proof1-new}, we obtain
\begin{eqnarray}\label{mid-1}
\|w_{i}\|_{b(\Omega)}^2
&\le&\tilde{C}M^2C_I(1+2\Lambda)\|w_{\bar{\Psi}^{(i)}}\|_{b(\Omega\backslash\Omega^{(i)}_{k-2,H})}^2.
\end{eqnarray}

Let $\Omega^{(i,o)}_{k,H} =\Omega\backslash \Omega^{(i)}_{k,H} $, $\Omega^{(i,H)}_{k-1,1}=\Omega^{(i)}_{k,H}\backslash \Omega^{(i)}_{k-2,H}$, use \eqref{b}, the definition of $I_h$, \eqref{chik-1} and \eqref{psi-sub2}, we get
\begin{eqnarray*}
\|w_{\bar{\Psi}^{(i)}}\|_{b(\Omega^{(i,o)}_{k,H})}^2
&=& b(w_{\bar{\Psi}^{(i)}},I_h((1-\chi_i^{k-1})w_{\bar{\Psi}^{(i)}}))_{\Omega}-a(w_{\bar{\Psi}^{(i)}},I_h((1-\chi_i^{k-1})w_{\bar{\Psi}^{(i)}}))_{\Omega^{(i,H)}_{k-1,1}}\nonumber\\
&-&\sum\limits_{\Omega_l\subset\Omega^{(i,H)}_{k-1,1}}
s_l(\pi_lw_{\bar{\Psi}^{(i)}},\pi_l(I_h((1-\chi_i^{k-1})w_{\bar{\Psi}^{(i)}})))\label{10}\\
&=& s_i(w_{\Phi^{(i)}_{\pi}},\pi_i(I_h((1-\chi_i^{k-1})w_{\bar{\Psi}^{(i)}})))-a(w_{\bar{\Psi}^{(i)}},I_h((1-\chi_i^{k-1})w_{\bar{\Psi}^{(i)}}))_{\Omega^{(i,H)}_{k-1,1}}\nonumber\\
&-&\sum\limits_{\Omega_l\subset\Omega^{(i,H)}_{k-1,1}}
s_l(\pi_lw_{\bar{\Psi}^{(i)}},\pi_l(I_h((1-\chi_i^{k-1})w_{\bar{\Psi}^{(i)}}))).\nonumber
\end{eqnarray*}

From this and using \eqref{chik-1}, \eqref{Ih-bound}, \eqref{Ih-bound-s}, the boundedness of $\pi_i$, \eqref{inequality-Ih-s}, Schwarz inequality, \eqref{chi-i-k-estimate} and  \eqref{lemma2-proof1}, let $\hat{C}=(\frac{3}{2}\tilde{C}M^2+2)C_I$, we derive
\begin{eqnarray}
\|w_{\bar{\Psi}^{(i)}}\|_{b(\Omega^{(i,o)}_{k,H})}^2
&\le&\frac{1}{2}C_{I_a}(|w_{\bar{\Psi}^{(i)}}|^2_{a(\Omega^{(i,H)}_{k-1,1})} +|(1-\chi_i^{k-1})w_{\bar{\Psi}^{(i)}}|^2_{a(\Omega^{(i,H)}_{k-1,1})})\nonumber\\
&+&\frac{1}{2}C_{I_s}\sum\limits_{\Omega_l\subset\Omega^{(i,H)}_{k-1,1}}
(\|\pi_lw_{\bar{\Psi}^{(i)}}\|^2_{s_l}+\|w_{\bar{\Psi}^{(i)}}\|^2_{s_l})\nonumber\\
&\le&\frac{1}{2}C_{I_a}((1+\tilde{C}M^2)|w_{\bar{\Psi}^{(i)}}|^2_{a(\Omega^{(i,H)}_{k-1,1})} + \tilde{C}M^2\sum\limits_{\Omega_l\subset\Omega^{(i,H)}_{k-1,1}}\|w_{\bar{\Psi}^{(i)}}\|_{s_l}^2)\nonumber\\
&+&\frac{1}{2}C_{I_s}\sum\limits_{\Omega_l\subset\Omega^{(i,H)}_{k-1,1}}
(\|\pi_lw_{\bar{\Psi}^{(i)}}\|^2_{s_l}+\|w_{\bar{\Psi}^{(i)}}\|_{s_l}^2)\nonumber\\
&\leq& (\frac{3}{2}C_{I_a}\tilde{C}M^2+2C_{I_s})(1+\Lambda)(|w_{\bar{\Psi}^{(i)}}|_{a(\Omega^{(i,H)}_{k-1,1})}^2+\sum\limits_{\Omega_l\subset\Omega^{(i,H)}_{k-1,1}}\|\pi_lw_{\bar{\Psi}^{(i)}}\|_{s_l}^2)\nonumber\\
&=& \hat{C}(1+\Lambda)\|w_{\bar{\Psi}^{(i)}}\|_{b(\Omega^{(i,H)}_{k-1,1})}^2.\label{wpsi-1-2}
\end{eqnarray}

It is easy to see that \eqref{wpsi-1-2} is equivalent to
\begin{eqnarray*}
(\hat{C}(1+\Lambda))^{-1}\|w_{\bar{\Psi}^{(i)}}\|_{b(\Omega\backslash \Omega^{(i)}_{k,H})}^2\le \|w_{\bar{\Psi}^{(i)}}\|_{b(\Omega^{(i)}_{k,H}\backslash \Omega^{(i)}_{k-2,H})}^2.
\end{eqnarray*}
We add $\|w_{\bar{\Psi}^{(i)}}\|_{b(\Omega \backslash \Omega^{(i)}_{k,H} )}^2$ on both sides of the above equation and obtain
\begin{eqnarray}\label{lemma3-2-new}
\|w_{\bar{\Psi}^{(i)}}\|_{b(\Omega\backslash \Omega^{(i)}_{k,H})}^2\leq E^{-1}\|w_{\bar{\Psi}^{(i)}}\|_{b(\Omega \backslash \Omega^{(i)}_{k-2,H} )}^2,
\end{eqnarray}
where $E = ((\hat{C}(1+\Lambda))^{-1}+1) > 1$.

Applying \eqref{lemma3-2-new} recursively, we have
\begin{eqnarray}\label{lemma4-0-new}
\|w_{\bar{\Psi}^{(i)}}\|_{b(\Omega\backslash\Omega^{(i)}_{k,H})}^2 \leq E^{-m} \|w_{\bar{\Psi}^{(i)}}\|_{b(\Omega)}^2,~~k=2m,~m\geq 1.
\end{eqnarray}

Combining \eqref{mid-1} and \eqref{lemma4-0-new}, we get \eqref{lemma2-4-new}.
  \end{proof}

In order to estimate $\|w_{\bar{\Psi}^{(i)}}\|_{b(\Omega)}^2$, we express the following assumption firstly.

\begin{assumption}\label{assump-3}
Assume that there exists a positive constant $C_{p}$ which satisfies
\begin{eqnarray}\label{Poincare}
\Sigma_{l=1}^N\|\pi_l u_0\|_{s_l}^2\lesssim C_{p}|u_0|_{a(\Omega)}^2,~~~~\forall u_0\in V_0,
\end{eqnarray}
where the constant in ``$\lesssim$'' is independent of the mesh and model parameters.
\end{assumption}

In an analogous way to lemma 2 of \cite{WCK2020}, we can obtain
\begin{lemma}\label{lemma-8}
For all $v^{(i)}_{aux}\in V^{(i)}_{aux}$, there exists a function $z\in V_h$ such that
\begin{eqnarray}\label{lemma6-7-2-assump-new-1}
\pi_i z &=& v^{(i)}_{aux}, ~~~~ i=1,\cdots,N,\\\label{lemma6-7-2-assump-new-2}
\|z\|_{b(\Omega)}^2 &\leq& (1+\Lambda^{-1}) \Sigma_{i=1}^N \|v^{(i)}_{aux}\|_{s_i}^2.
\end{eqnarray}
\end{lemma}

\begin{proof}
For any $v^{(i)}_{aux}\in V^{(i)}_{aux}$, let $z:=I_h v_{aux}^{(i)}\in V_h(\Omega)$,
it is easy to know that
\begin{eqnarray}\label{chi-v-new}
z|_{\bar{\Omega}_i}=v_{aux}^{(i)},~~\pi_i z &=& \Sigma_{j=1}^{l_i}s_i(z|_{\Omega_i},\phi_j^{(i)})\phi_j^{(i)}.
\end{eqnarray}

Now we prove that the above function $z$ satisfies \eqref{lemma6-7-2-assump-new-1} and \eqref{lemma6-7-2-assump-new-2}.

Let
\begin{eqnarray}\label{vaux1-new}
v_{aux}^{(i)}:= \Sigma_{j=1}^{l_i}c_j\phi_j^{(i)},~~c_j\in \mathbb{C},~~j=1,\cdots,l_i.
\end{eqnarray}

By \eqref{chi-v-new}, \eqref{vaux1-new} and \eqref{delta-ij}, we have
\begin{eqnarray*}
s_i(z|_{\Omega_i},\phi_j^{(i)})
= s_i(v_{aux}^{(i)},\phi_j^{(i)})
= s_i(\Sigma_{k=1}^{l_i}c_k \phi_k^{(i)},\phi_j^{(i)})
= \Sigma_{k=1}^{l_i}c_k s_i( \phi_k^{(i)},\phi_j^{(i)})
= c_j.
\end{eqnarray*}
From this and \eqref{chi-v-new}, \eqref{vaux1-new}, we obtain \eqref{lemma6-7-2-assump-new-1}.

Furthermore, using \eqref{chi-v-new}, we conclude that
\begin{eqnarray*}
|z|_{a(\Omega_i)}^2 &=&  |z|_{\Omega_i}|_{a(\Omega_i)}^2=|v_{aux}^{(i)}|_{a(\Omega_i)}^2\leq \Lambda^{-1} \|v_{aux}^{(i)}\|_{s_i}^2.
\end{eqnarray*}

We add $\sum_{i=1}^N\|\pi_iz\|_{s_i}^2$ on both sides of the above equation, and note the definition of $\|\cdot\|_{b(\Omega)}$ and \eqref{lemma6-7-2-assump-new-1}, then \eqref{lemma6-7-2-assump-new-2} holds.
\end{proof}

From LEMMA \ref{lemma-8} and \eqref{Poincare}, we obtain
\begin{lemma}\label{assump-6}
If ASSUMPTION \ref{assump-3} holds, we have
\begin{eqnarray}\label{lemma6-7-a}
 \Sigma_{i=1}^N\|w_{\bar{\Psi}^{(i)}}\|_{b(\Omega)}^2 &\lesssim&  2(1+C_p)|u_{0,\bar{\psi}}|_{a(\Omega)}^2.
\end{eqnarray}
\end{lemma}
\begin{proof}
By using \eqref{psi-sub2} and \eqref{lemma-0}, we have
\begin{eqnarray}\label{psi-b-s}
b(w_{\bar{\Psi}^{(i)}},v)= s_i(w_{\Phi_{\pi}^{(i)}},\pi_iv),~~\forall v\in V_h.
\end{eqnarray}
From this and noting that $u_{0,\bar{\psi}}=\sum_{i=1}^N w_{\bar{\Psi}^{(i)}}$, we get
\begin{eqnarray}\label{lemma6-7-1-new-new}
b(u_{0,\bar{\psi}},v)= \Sigma_{i=1}^Ns_i(w_{\Phi_{\pi}^{(i)}},\pi_iv),~~\forall v\in V_h.
\end{eqnarray}

By using LEMMA \ref{lemma-8}, we see that for $w_{\Phi_{\pi}^{(i)}}\in V^{(i)}_{aux}$ there exists $z\in V_h$ such that
\begin{eqnarray}\label{lemma6-7-2-new}
\pi_i z = w_{\Phi_{\pi}^{(i)}}, ~\forall i=1,\cdots,N,~~\|z\|_{b(\Omega)}^2\leq (1+\Lambda^{-1}) \Sigma_{i=1}^N \|w_{\Phi_{\pi}^{(i)}}\|_{s_i}^2.
\end{eqnarray}

Set $v=z$ in \eqref{lemma6-7-1-new-new}, and by \eqref{lemma6-7-2-new}, we have
\begin{eqnarray*}
\Sigma_{i=1}^N\|w_{\Phi_{\pi}^{(i)}}\|_{s_i}^2 = b(u_{0,\bar{\psi}},z)\leq \|u_{0,\bar{\psi}}\|_{b(\Omega)} \|z\|_{b(\Omega)}\le  \|u_{0,\bar{\psi}}\|_{b(\Omega)}
\big((1+ \Lambda^{-1})\Sigma_{i=1}^N \|w_{\Phi_{\pi}^{(i)}}\|_{s_i}^2\big)^{\frac{1}{2}},
\end{eqnarray*}
namely,
\begin{eqnarray}\label{lemma6-7-0-new}
\Sigma_{i=1}^N\|w_{\Phi_{\pi}^{(i)}}\|_{s_i}^2 \leq (1+\Lambda^{-1})\|u_{0,\bar{\psi}}\|_{b(\Omega)}^2.
\end{eqnarray}

Furthermore, using the definition of $\|\cdot\|_{b(\Omega)}$ and \eqref{Poincare}, we can easily obtain
\begin{eqnarray}\label{lemma6-7-4-new}
\|u_{0,\bar{\psi}}\|_{b(\Omega)}^2 \lesssim (1+C_{p})|u_{0,\bar{\psi}}|_{a(\Omega)}^2.
\end{eqnarray}

Set $v=w_{\bar{\Psi}^{(i)}}$ in \eqref{psi-b-s} and noting the boundedness of $\pi_i$, we conclude
\begin{eqnarray*}
\|w_{\bar{\Psi}^{(i)}}\|_{b(\Omega)}^2
&\leq& \|w_{\Phi_{\pi}^{(i)}}\|_{s_i}\|\pi_iw_{\bar{\Psi}^{(i)}}\|_{s_i}
\le \|w_{\Phi_{\pi}^{(i)}}\|_{s_i}\|w_{\bar{\Psi}^{(i)}}\|_{{b(\Omega)}},
\end{eqnarray*}
from this and combining \eqref{lemma6-7-0-new} and \eqref{lemma6-7-4-new} give \eqref{lemma6-7-a}.
\end{proof}

Combining \eqref{lemma5-0}, \eqref{lemma2-4-new} and \eqref{lemma6-7-a}, and $a(u_{0,\bar{\psi}},u_{0,\bar{\psi}})\le a(u,u)$, we conclude
\begin{eqnarray}\label{lemma7-0}
\|w\|_{b(\Omega)}^2 &\lesssim&  C_0|u|_{a(\Omega)}^2,
\end{eqnarray}
where
\begin{eqnarray}\label{C0}
C_0&\le& \tilde{C}_3 \Lambda^2(2k+1)^d E^{1-\frac{k}{2}} (1+C_{p}),~~and~~\tilde{C}_3=2\tilde{C}_1\tilde{C}_2.
\end{eqnarray}

Substituting \eqref{lemma7-0} into \eqref{CondA2}, we have
\begin{eqnarray}\label{CondA2-new}
\Sigma_{i=1}^N|u_{i,\zeta}|_{a(\Omega'_i)}^2 + |u^{(k)}_{0,\zeta}|_{a(\Omega)}^2\le  C_2|u|_{a(\Omega)}^2,
\end{eqnarray}
where $C_2\le\tilde{C}_4(\Lambda+C_0), \tilde{C}_4=6+4C_{I}\tilde{C}M$.

Combining \eqref{max-ecnomic-proof2}, \eqref{CondA2-new} and Theorem 2.2 in \cite{HX2007}, we get 
\begin{theorem}\label{theorem-1-new}
If ASSUMPTION \ref{assump-1-strong}, \ref{assump-2} and \ref{assump-3} all hold and $k\geq 2$, we have
\begin{eqnarray}\label{theorem-ec}
\kappa(B^{-1}_{k,\bar{\psi}}A_h)\lesssim  2MC_2,
\end{eqnarray}
where the constant in ``$\lesssim$'' is independent of the mesh and model parameters.
\end{theorem}

\begin{corollary}\label{corollary-1}
Under the assumption of THEOREM \ref{theorem-1-new} and $C_p$ is a positive constant, then $\kappa(B^{-1}_{k,\bar{\psi}}A_h)$ is bounded by a constant only dependent on $\Lambda$ and $M$.
\end{corollary}

\begin{remark}
Similar to the proof of $\kappa(B^{-1}_{k,\bar{\psi}}A_h)$, and through more elaborate analysis, we can obtain the estimate of $\kappa(B^{-1}_{k,\psi}A_h)$. Due to space limitation, we will not describe the detailed proof process.
\end{remark}

In the following, we will apply the above algorithmic and theoretical frameworks for Hermitian positive definite system to two typical models.

\section{TL-OS preconditioners for second order elliptic problem}\label{Poisson}
Let $\Omega$ be a bounded convex polyhedral region, we consider the following second order elliptic problem: find $u\in V(\Omega):=H^1_0(\Omega)$ such that
\begin{eqnarray}\label{model-elliptic}
a^{(1)}(u,v)=f(v),~~\forall v\in V(\Omega),
\end{eqnarray}
where $f\in L^2(\Omega)$ and real symmetric and positive definite bilinear functional
\begin{eqnarray*}
a^{(1)}(u,v):=\int_{\Omega}\rho(\bx)\nabla u \cdot \nabla v dx,~~u,v\in V(\Omega),
\end{eqnarray*}
positive function $\rho(\bx)\in L^\infty(\Omega)$ may be highly heterogeneous with very high contrast.

Let $\mathcal{T}_h$ be a simplex partition of $\Omega$ and $\mathcal{X}_h$ be the set of partition node, $\{\Omega_i\}^N_{i=1}$ and $\{\Omega'_i\}^N_{i=1}$ are separately the nonoverlapping and overlapping domain decomposition mentioned in section \ref{frameworks}. Define $V_{h,1}$ as the space consisting of continuous piecewise linear functions associated with $\mathcal{T}_h$  which vanishes on $\partial\Omega$.
Then the discrete system of \eqref{model-elliptic} corresponding to $V_{h,1}$ can be written as
\begin{eqnarray}\label{sys-elliptic}
A_{h,1} u_{h,1}=f_{h,1},~~~~u_{h,1}\in V_{h,1}.
\end{eqnarray}


\subsection{TL-OS preconditioner for linear finite element equation}\label{Poisson-TL-OS}
Define the partition of unity function $\theta_i(x)\in V_{h,1}$ with its values at any $\bx\in \mathcal{X}_h$ being
\begin{eqnarray*}
\theta_i(\bx)=
\left\{
\begin{array}{lll}
\frac{1}{|\mathcal{N}_{\bx}|}, & if ~ \bx\in \Omega'_i ~or~\bx\in \partial\Omega'_i\cap \partial\Omega \\
0,& otherwise
\end{array}
\right.~~i=1,\cdots,N,
\end{eqnarray*}
where $|\mathcal{N}_{\bx}|$ denote the number of elements in $\mathcal{N}_{\bx}=\{j: \bx\in \Omega'_j~or~\bx\in \partial\Omega'_j\cap \partial\Omega\}$. It is easy to see that $\{\theta_i(\bx)\}_{i=1}^N$ satisfy \eqref{theta-def}.

Let $\mathcal{S}^{(i)}$ be the index set defined by \eqref{S-i-M}, and we introduce the bilinear functional
\begin{eqnarray}\label{si-def0}
s^{(1)}_i(u,w)&=&\int_{\Omega_i} \rho(\bx) \Sigma_{l\in \mathcal{S}^{(i)}}|\nabla \theta_l(\bx)|^2 uwdx.
\end{eqnarray}

Notice that $\theta_j(\bx)$ is supportable, then we can check that $s^{(1)}_i(\cdot,\cdot)$ is a symmetric and positive definite bilinear functional which satisfies
\begin{eqnarray*}
a^{(1)}_i(\theta_j u,\theta_j u)
&\le&2(a^{(1)}_i(u, u)+s^{(1)}_i(u,u)),~~\forall u\in V_{h,1},
\end{eqnarray*}
namely, the ASSUMPTION \ref{assump-1-strong} holds.

Using $s^{(1)}_i(\cdot,\cdot)$ and $a^{(1)}_i(\cdot,\cdot)$, we can introduce the corresponding coarse space $V_{0,1}:=V_0$ defined by \eqref{zjb-def}.
Let the operators associated with $V_{h,1}$ and bilinear form $a^{(1)}(\cdot,\cdot)$ be $A_{j,1}, A_{0,1}$, $\Pi_{j,1}$ and $\Pi_{0,1}$(see the definitions \eqref{tildeA-def}, \eqref{tildeA0-def}, etc.) respectively,
and base on the general formula \eqref{B-def-new}, we can construct the TL-OS preconditioner for system \eqref{sys-elliptic} as follows
\begin{eqnarray}\label{B-def-elliptic}
B^{-1}_1 = \Sigma_{j=1}^N \Pi_{j,1} A_{j,1}^{-1}\Pi_{j,1}^*+\Pi_{0,1} A_{0,1}^{-1}\Pi_{0,1}^*.
\end{eqnarray}

For any $w\in C(\bar{\Omega})$, we define an interpolation $I_h: C(\bar{\Omega})\rightarrow V_{h,1}$ such that
\begin{eqnarray*}
I_h w\in V_{h,1}~\mbox{~and~}~I_hw(\bx) = w(\bx), ~~\bx\in \mathcal{X}_h.
\end{eqnarray*}
From the basic theory of finite element method, we know that when $\mathcal{T}_h$ is quasi-uniforming, there exists positive constants $C_{I_a}$ and $C_{I_s}$ which satisfy
\begin{eqnarray*}
a^{(1)}(I_h u, I_h u)_{\tau}&\le&  C_{I_a}a^{(1)}(u, u)_{\tau},~~\forall \tau\in \mathcal{T}_h,\\
\Sigma_{l=1}^N s^{(1)}_l(I_h u, I_h u)&\le& C_{I_s} \Sigma_{l=1}^N s^{(1)}_l( u, u),
\end{eqnarray*}
where $u\in H^1_0(\Omega)$. Namely, the ASSUMPTION \ref{assump-2} holds.

Then by THEOREM \ref{theorem-general}, we have
\begin{theorem}\label{theorem-elliptic}
If $\mathcal{T}_h$ is quasi-uniforming, the precondtioner $B^{-1}_1$ defined by \eqref{B-def-elliptic} satisfies
\begin{eqnarray*}
\kappa(B^{-1}_1A_{h,1})\le C,
\end{eqnarray*}
where the positive constant $C$ depends only on $\Lambda$ and $M$.
\end{theorem}

Because the calculation of coarse basis (see \eqref{psi-sub2-new} or \eqref{psi-sub2}) takes up a large proportion in the preconditioner formula defined in \eqref{B-def-elliptic}, 
%
%
and the coefficient matrices of \eqref{psi-sub2-new} and \eqref{psi-sub2} are separately sparse and block dense, we can use fast solvers with lower computational complexity, e.g. AMG, for solving \eqref{psi-sub2-new}, but usually direct solvers with higher computational complexity for solving \eqref{psi-sub2}. Therefore we often choose \eqref{psi-sub2-new} as the coarse basis formula. Further experiments also show that the accuracy of solving the problem \eqref{psi-sub2-new} does not need to be too high.

Let $\Omega=(0,1)^3$, $f(\bx)=3\pi^2\sin(\pi x_1)\sin(\pi x_2)\sin(\pi x_3)$, we consider the following coefficient distributions in second order elliptic problem \eqref{model-elliptic}
\begin{eqnarray}\label{coef-1}
    \rho(\bx)=\left\{\begin{array}{cc}
      10^{\mu_1}, &~~~\bx\in(\frac{1}{4},\frac{2}{4})\times (0,\frac{2}{4})\times (0,\frac{1}{4})\cup(\frac{1}{4},\frac{2}{4})^3 \\
      10^{\mu_2}, &~~~\bx\in(\frac{1}{4},\frac{2}{4})^3,\\
      1,&\text{otherwise}
    \end{array}\right., ~~\mu_1,\mu_2\ge 0
  \end{eqnarray}
and
\begin{eqnarray}\label{coef-2}
\rho(\bx)|_{\tau}\in(10^{-\frac{\mu}{2}},10^{\frac{\mu}{2}}),~~\forall \tau\in\mathcal{T}_h.
\end{eqnarray}
Denote the problem \eqref{model-elliptic} corresponding to $\mu_1=\mu_2=0$, $\mu_1=\mu, \mu_2=0$ and $\mu_1=0, \mu_2=\mu$ in \eqref{coef-1} as $Model_1$-$Model_3$ respectively, and the problem \eqref{model-elliptic} related to \eqref{coef-2} as $Model_4$.

%
%
%

The resulting system \eqref{sys-elliptic} corresponding to $Model_1$-$Model_4$ is solved by the PCG method with preconditioner $B^{-1}_1$ defined by \eqref{B-def-elliptic}. Let $\mathcal{T}_h=n(m)$ be a conforming triangulation with subdomain size $H=1/n$ and mesh size $h=1/(nm)$(see \cite{HSW2010} for more details). Fig. \ref{grid-nm} and Fig. \ref{Fig1:nonoverlapping-subdomain} show  $\mathcal{T}_h=5(5)$ and some nonoverlapping subdomain $\Omega_i$ in two dimension. In addition, extending each $\Omega_i$ with width $\delta=lh$ where $l$ denotes the number of layers in $\mathcal{T}_h$, we get the overlapping subdomain $\Omega'_{i}:=\Omega^{(i)}_{l,h}$, Fig. \ref{Fig2:overlapping-parts} lists some $\Omega^{(i)}_{1,h}$. Set $\Lambda=1+\log H/h$, the PCG method is stopped when the relative residual is reduced by the factor of $tol=10^{-6}$, and we call AMG for solving the variational problem \eqref{psi-sub2-new} related to the coarse basis. 
\begin{figure}[h]
\centering
\subfigure[]{\label{grid-nm}
\includegraphics[width=1.0in]{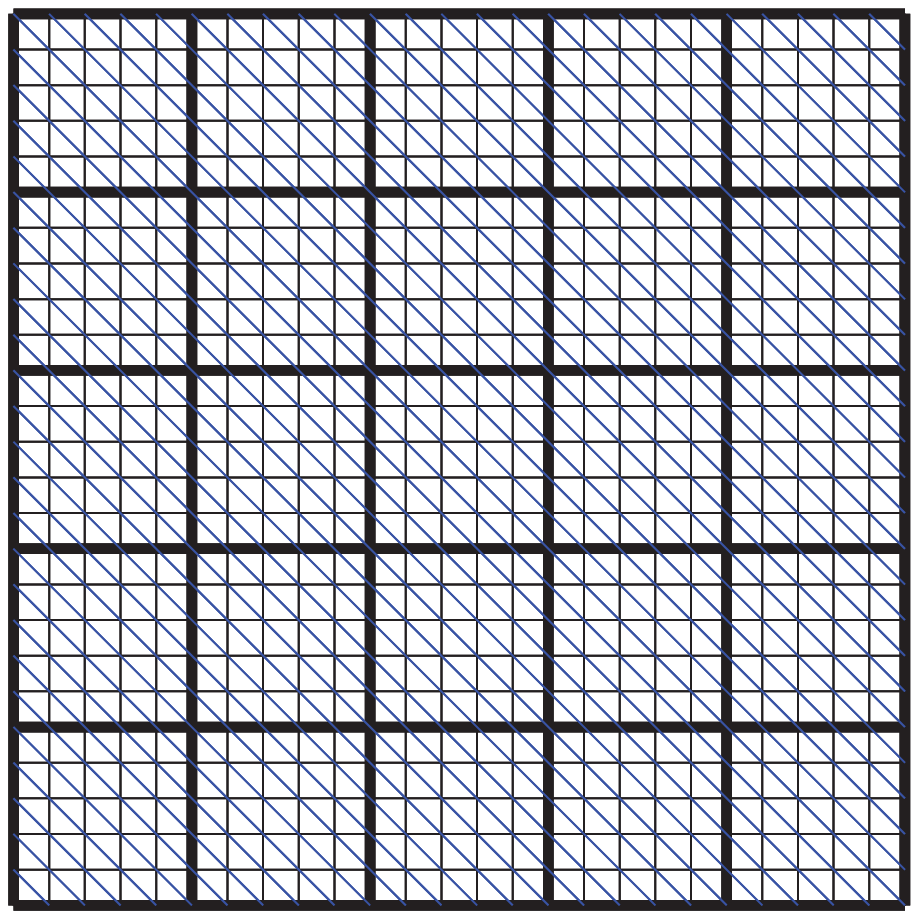}}
\hspace{0.001in}
\subfigure[]{\label{Fig1:nonoverlapping-subdomain}
\includegraphics[width=1.0in]{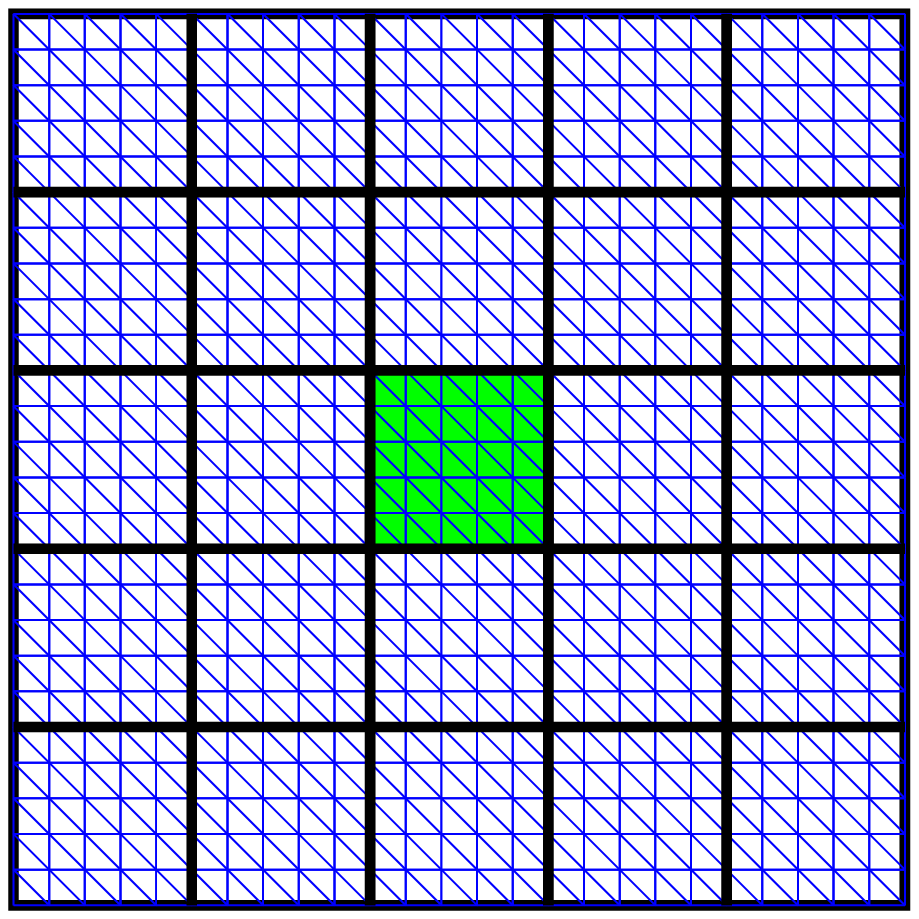}}
\hspace{0.001in}
\subfigure[]{ \label{Fig2:overlapping-parts}
\includegraphics[width=1.0in]{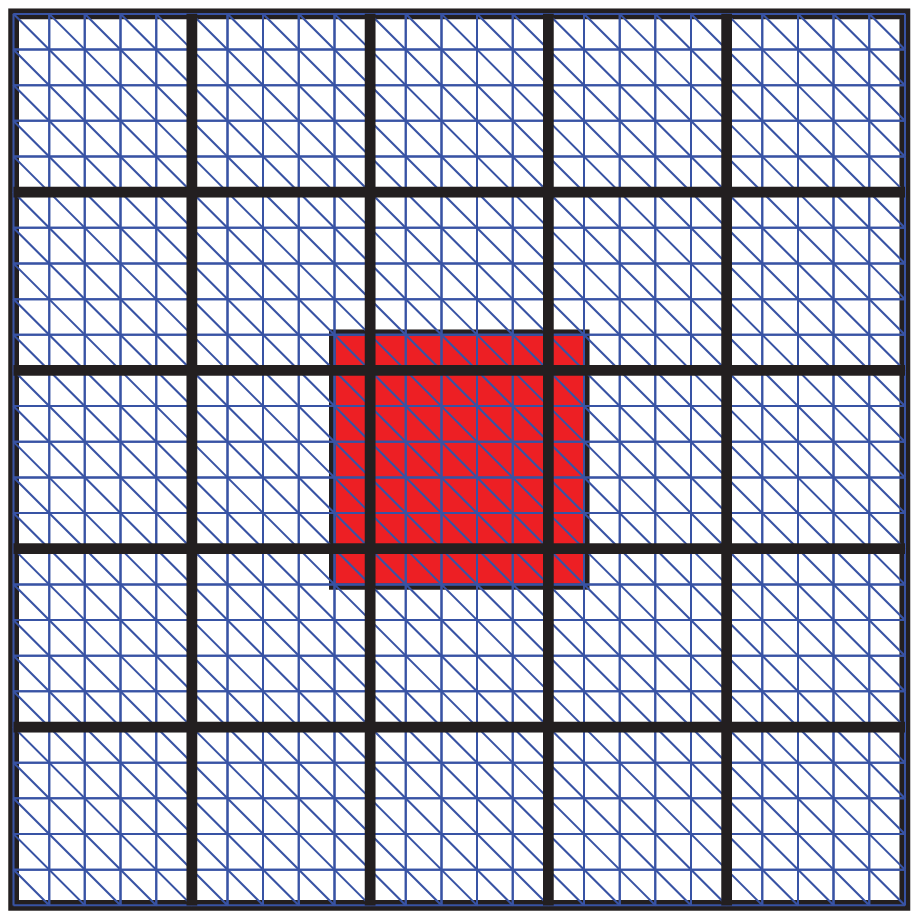}}
\hspace{0.001in}
{\caption{(a)~$\mathcal{T}_h=5(5)$, (b)~The green domain is $\Omega_i$, (c)~The red domain is $\Omega'_{i}:=\Omega^{(i)}_{l,h}$ with $l=1$ .}
\label{Fig:submesh-2}}
\end{figure}
The dependence of the iteration counts on $l$ is listed in TABLE \ref{A-48-k-new} with fixed $n(m)=4(8)$ and exponent $\mu=6$ in the coefficient $\rho(\bx)$, the data in and out of brackets indicate the number of TL-OS preconditioned PCG iterations when the AMG iterative control accuracy $tol_A$ is chosen as $10^{-10}$ and $10^{-1}$ respectively, it can be seen that they are nearly the same. Therefore, we take $tol_A=10^{-1}$ in the following numerical experiments. we also see from TABLE \ref{A-48-k-new} that the iteration numbers are nearly independent of $l$. Then for a fixed $l=1$, TABLE \ref{rho-distribution-nm-468} reported the results with varying $n(m)$ and $\mu$. We observe that the number of iterations is almost independent of the mesh size, jump coefficient distribution and jump range.

Noticing $M\le 27$, we find from the above results that the iteration counts are independent of the overlapping width $\delta$(or $l$), mesh size $h$ and coefficient $\rho(\bx)$ when $\Lambda=1+H/h$ is fixed, and the iteration counts increase slightly when $\Lambda$ grows which verify THEOREM \ref{theorem-elliptic}.


\begin{table}[h]
\footnotesize \caption{~$n(m)=4(8)$, $\mu=6$ }
\label{A-48-k-new} \centering
\begin{tabular}{|ccccc|}
  \hline
  ~~$l$  ~& $Model_1$~~&$Model_2$~~ & $Model_3$~~& $Model_4$~~ \\
 \hline
        1     &  17(17)  & 17(17) &17(17) &19(19)       \\
        2     &  16(16)  & 17(17) &17(17) &19(19)    \\
        3     &  18(17)  & 18(18) &19(19) &18(18)     \\
        4     &  16(15)  & 17(17) &17(17) &18(18)  \\
        \hline
\end{tabular}
\end{table}

\begin{table}[h]
\footnotesize\centering\caption{$l=1$}
\label{rho-distribution-nm-468}\vskip 0.1cm
\begin{tabular}{{|c|cccc|ccc|ccc|}}\hline
\multirow{2}{*}{$n(m)$}& & \multicolumn{3}{c|}{$Model_2$} & \multicolumn{3}{c|}{$Model_3$} & \multicolumn{3}{c|}{$Model_4$}  \\\cline{2-11}
                      & $\mu=$ &4     & 6  & 8  &4  &6  &8  &4  &6    & 8\\       \hline
4(4)                  &       & 17  &17  &17  &17 &17 &17 &18   &18   &19     \\  
4(6)                  &       & 17  &17  &17  &17 &17 &17 &19   &20   &20   \\ \hline\hline
8(4)                  &       & 17  &16  &16  &16 &18 &19 &19   &19   &19  \\ 
8(6)                  &       & 16  &16  &16  &16 &16 &16 &19   &19   &19         \\ \hline
\end{tabular}
\end{table}

In next subsection, we will discuss the economical TL-OS precondioners. 

\subsection{Economical precondioners and numerical experiments}\label{Poisson-economical}

Similar to the introduction of TL-OS preconditioner \eqref{B-def-elliptic} in subsection \ref{Poisson-TL-OS}, and using the construction framework of the two economical precondionters in subsection \ref{framework-economical}, we can give two economical TL-OS precondionters for \eqref{sys-elliptic} as follows
\begin{eqnarray}\label{precond-ec-Poisson}
(B^{(1)}_{k,\zeta})^{-1} = \Sigma_{j=1}^N \Pi_{j,1} A_{j,1}^{-1}\Pi_{j,1}^*+\Pi^{k,\zeta}_{0,1} (A^{k,\zeta}_{0,1})^{-1}(\Pi^{k,\zeta}_{0,1})^*,~~\zeta=\psi, \bar{\psi},
\end{eqnarray}
where $\Pi^{k,\zeta}_{0,1}: V^{k,\zeta}_{0,1}\rightarrow V_{h,1}$ are identical lifting operators, the coarse space operator $A^{k,\zeta}_{0,1}: V^{k,\zeta}_{0,1}\rightarrow V^{k,\zeta}_{0,1}$ satisfies $(A^{k,\zeta}_{0,1}u,v)=a^{(1)}(u,v), \forall u,v\in V^{k,\zeta}_{0,1}$,
and $V^{k,\zeta}_{0,1}\subset V_{h,1}$ here is the corresponding economical coarse space.

In the following, we will estimate the condition number $\kappa((B^{(1)}_{k,\bar{\psi}})^{-1}A_h)$ by using the theoretical framework in subsection \ref{framework-economical}. Noting that in subsection \ref{Poisson-TL-OS}, we have checked that the ASSUMPTION \ref{assump-1-strong} and \ref{assump-2} hold, so we need only to prove that the ASSUMPTION \ref{assump-3} holds now.

Firstly, we assume that for a given constant $C_{\rho}$, the triangulation $\mathcal{T}_h$ is fine enough to 
satisfy that
\begin{eqnarray}\label{max-min-rho}
\frac{\max_{\bx\in \tau} \rho(\bx)}{\min_{\bx\in \tau} \rho(\bx)}\le C_{\rho},~~\forall \tau\in \mathcal{T}_h.
\end{eqnarray}

Let $\Omega_{i}^{INT}=\{\bx\in \Omega\big| \theta_i(\bx)=1\}$, Fig. \ref{subdomain-interior} shows
the schematic figure of $\Omega_{i}^{INT}$ in two dimension.
Notice that $\|\cdot\|^2_{s_i}:=s^{(1)}_i(\cdot,\cdot)$, and using \eqref{max-min-rho}, we can easily obtain
\begin{eqnarray*}
C_p=C_{\rho}^2\frac{\max_{l} (\max_{\tau\in \Omega_l \backslash\Omega_l^{INT}}(\min_{\bx\in \tau} (\rho(\bx)\sum_{m\in n(l)}|\nabla\theta_m(\bx)|^2)))}{\min_{\tau\in \mathcal{T}_h}(\max_{\bx\in\tau}\rho(\bx))}.
\end{eqnarray*}

From this, we see that $C_{p}$ depends on the coefficient $\rho(\bx)$, e.g. for $Model_3$ in one dimension, we can prove $C_{p}=O(10^{\mu})$.
As a result, by THEOREM \ref{theorem-1-new} and \eqref{C0}, in order to get an appropriate condition number, the number of layers $k$ in $\Omega_{k,H}^{(i)}$ should be increased with the increase of $\mu$ which raises the computational cost of the coarse basis. Through an in-depth analysis, we find that the above theoretical results can be substantially improved when $\rho(\bx)$ satisfies a certain distribution. Now, we introduce the special distribution condition which the coefficient $\rho(\bx)$ satisfies
\begin{assumption}\label{cond-rho-Omega}
For a given subdomain $\Omega_i$ which does not touch the boundary $\partial\Omega$, we assume that
there exists an index set $\mathcal{S}_i=\{i,i_1,\cdots,i_{\check{n}_i}\}\subset \{1,\cdots,N\}$ such that $\bar{\Omega}_i\cup\bar{\Omega}_{i_1}\cup\cdots\cup \bar{\Omega}_{i_{\check{n}_i}}(\check{n}_i\geq 1)$ satisfies  $\bar{\Omega}_i\cup\bar{\Omega}_{i_1}\neq \emptyset$, $\bar{\Omega}_{i_j}\cap \bar{\Omega}_{i_{j+1}}\neq \emptyset(j=1,\cdots, \check{n}_i-1)$,
$\partial\Omega_{i_{\check{n}_i}}\cap\partial\Omega\neq \emptyset$. And there exists a constant $\check{C}>0$ such that
$\frac{\max_{j\in\mathcal{S}_i}\max_{\bx\in\bar{\Omega}_j}\rho_j}{\min_{j\in\mathcal{S}_i}\min_{\bx\in\bar{\Omega}_j}\rho_j}\leq \check{C}$
where $\rho_j=\rho|_{\Omega_j}$ and $\check{C}>0$ is independent of the jump of $\rho(\bx)$ in $\Omega$.
\end{assumption}

\begin{lemma}
Under ASSUMPTION \ref{cond-rho-Omega}, and for any $u_0\in V_{0,1}$, we have
\begin{eqnarray}\label{lemma6-7-a-new}
\Sigma_{l=1}^N\| \pi_lu_0\|_{s_l}^2 \le  C_p|u_0|_{a(\Omega)}^2,
\end{eqnarray}
where the positive constant $C_p$ is independent of $\rho(\bx)$.
\end{lemma}

\begin{proof} Using the definition \eqref{si-def0} of $s^{(1)}_i(\cdot,\cdot)$ and noting that  $\sum_{m\in n(l)}|\nabla \theta_m|^2$ is constant in each $\tau$ in $\mathcal{T}_h$, we have
  \begin{eqnarray}\nonumber
  \Sigma_{l=1}^N\|\pi_lu_0\|_{s_l}^2&=&\Sigma_{l=1}^N\|u_0\|_{s_l}^2= \Sigma_{l=1}^N\Sigma_{\tau\in \Omega_l \backslash \Omega_l^{INT}}\int_{\tau} \rho(\bx)\Sigma_{m\in S^{(l)}}|\nabla \theta_m|^2 u_0^2dx \\
    &\leq& C(\theta)\Sigma_{l=1}^N \int_{ \Omega_l \backslash \Omega_l^{INT}} \rho(\bx) u_0^2dx, \label{lemma67220}
      \end{eqnarray}
where $C(\theta)=\Sigma_{l} (\Sigma_{\tau\in \Omega_l \backslash \Omega_l^{INT}} (\Sigma_{m\in S^{(l)}}|\nabla \theta_m|^2)|_{\tau})$.

Furthermore, by ASSUMPTION \ref{cond-rho-Omega} and Poincar$\acute{e}$ inequality, we obtain
\begin{eqnarray}\nonumber
\Sigma_{l=1}^N\int_{\Omega_l\backslash \Omega_l^{INT}} \rho(\bx) u_0^2dx
&\le&\Sigma_{l=1}^{\check{m}} \max_{k\in \mathcal{S}_l}\rho_k \int_{\check{\Omega}_l} u_0^2 dx \le\check{C} \Sigma_{l=1}^{\check{m}} \int_{\check{\Omega}_l} \rho(\bx) |\nabla u_0|^2 dx=\check{C}|u_0|_{a(\Omega)}^2.
\end{eqnarray}

Substituting this into \eqref{lemma67220} leads to \eqref{lemma6-7-a-new} with 
$C_p=\check{C} C(\theta)$.
\end{proof}

By the above lemma and THEOREM \ref{theorem-1-new}, we have
\begin{theorem}\label{theorem-2}
Assume that $\rho(\bx)$ satisfies ASSUMPTION \ref{cond-rho-Omega} and $k\ge 2$, we have
\begin{eqnarray*}
\kappa((B^{(1)}_{k,\bar{\psi}})^{-1}A_{h,1})\le C,
\end{eqnarray*}
where the positive constant $C$ is independent of $\rho(\bx)$.
\end{theorem}

For $Model_2$ that satisfies ASSUMPTION \ref{cond-rho-Omega} and $Model_3$ that does not satisfy ASSUMPTION \ref{cond-rho-Omega}, we present the numerical results of solving discrete system \eqref{sys-elliptic} by PCG method based on preconditioner $(B^{(1)}_{k,\bar{\psi}})^{-1}$ defined by \eqref{precond-ec-Poisson}, where we choose MUMPS to solve \eqref{psi-ec-II}, and the other parameters are selected the same as subsection \ref{Poisson-TL-OS}.

For a fixed $l=1$, we present the iteration number in TABLE \ref{ec-AS-468} as $n(m)$, $\mu$ and $k$(related to the computational domain $\Omega_{k,H}^{(i)}$ of the coarse basis) varying. We observe that the number of iterations corresponding to $Model_2$ does not depend on the jump range while $Model_3$ does, which confirm the estimate in THEOREM \ref{theorem-2}.
\begin{table}[h]
\footnotesize\centering\caption{$l=1$,$\mu=4,6,8$}
\label{ec-AS-468}\vskip 0.1cm
\begin{tabular}{{|c|cccc|ccc|}}\hline
\multirow{2}{*}{$n(m)$}&   & \multicolumn{3}{c|}{$Model_2$}  & \multicolumn{3}{c|}{$Model_3$}  \\\cline{2-8}
                       &$k=$  & 1        & 2   & 3   & 1   & 2   & 3\\       \hline
8(4)                   &      &24(21,21) &17(17,17)  &16(16,16) &31(36,43)   &18(18,30)  &17(18,28)  \\
12(4)                   &      &28(27,26) &18(17,17)  &16(16,16) &36(42,49)   &28(34,41)  &18(18,25) \\ \hline
\end{tabular}
\end{table}


In the following experiments, 
we apply PCG method with preconditioner $(B^{(1)}_{k,\psi})^{-1}$ defined by \eqref{precond-ec-Poisson} which possesses lower computational complexity for solving the system \eqref{sys-elliptic} associated with $Model_1-Model_4$, and AMG is used to solve \eqref{psi-ec-I}. The parameters are selected the same as subsection \ref{Poisson-TL-OS}.

In TABLE \ref{ec-A-nm-88},  for fixed $n(m)=8(8)$ and $\mu=6$, we investigate the varying of the iteration number as $l$ and $k$ changing. The results demonstrate that for $Model_i(i=1,2,4)$, the number of iterations almost does not depend on $l$ and $k$, but for $Model_3$, the number of iterations is stable at $k\ge2$.
\begin{table}
\footnotesize\centering\caption{$n(m)=8(8),\mu=6$}
\label{ec-A-nm-88}\vskip 0.1cm
\begin{tabular}{{|c|cccc|ccc|ccc|ccc|}}\hline
\multirow{2}{*}{$l$}&  & \multicolumn{3}{c|}{$Model_1$} & \multicolumn{3}{c|}{$Model_2$}  & \multicolumn{3}{c|}{$Model_3$} & \multicolumn{3}{c|}{$Model_4$}  \\\cline{2-14}
                    &$k=$ &1   & 2   & 3 &1  & 2   & 3&1  & 2   & 3 &1  & 2   & 3\\       \hline
1                   &   & 18 &17 &17 &18&17&17&35&19 &17 &22  &20 &20    \\  
2                   &   & 16 &16 &15 &19&17&16&34&21 &19 &21  &20 &20  \\
3                   &   & 19 &18 &18 &21&18&18&37&19 &19 &19  &18 &18 \\ 
4                   &   & 18 &17 &17 &20&18&18&35&19 &19 &19  &18 &18  \\ \hline
\end{tabular}
\end{table}

In TABLE \ref{ec-A-nm-468-mu-468}, for the case of $l=1$, we investigate the phenomena with the change of $n(m)$, $k$ and $\mu$, where the three data outside and inside the brackets represent the number of iterations when $\mu=4,6,8$, respectively. We can see from TABLE \ref{ec-A-nm-468-mu-468} that for $Model_2$ and $Model_4$, the number of iterations is almost independent of $n(m)$, $k$ and $\mu$, while $Model_3$ is weakly dependent on these factors. In addition, comparing with TABLE \ref{ec-AS-468}, we find that the iteration number of PCG method based on preconditioner $(B^{(1)}_{k,\psi})^{-1}$ is more stable than preconditioner $(B^{(1)}_{k,\bar{\psi}})^{-1}$ for $Model_3$.
\begin{table}[h]
\footnotesize\centering\caption{$l=1, \mu=4,6,8$}
\label{ec-A-nm-468-mu-468}\vskip 0.1cm
\begin{tabular}{{|c|ccc|cc|cc|}}\hline
\multirow{2}{*}{$n(m)$}&   & \multicolumn{2}{c|}{$Model_2$}  & \multicolumn{2}{c|}{$Model_3$} & \multicolumn{2}{c|}{$Model_4$}  \\\cline{2-8}
                       &$k=$  & 2        & 3   & 2   & 3   & 2   & 3\\       \hline
8(6)                   &      &17(16,17) &16(16,16)  &18(18,18) &17(18,18)   &20(20,20)  &19(19,19)  \\
8(8)                   &      &17(17,17) &17(17,17)  &18(19,19) &17(17,17)   &20(20,20)  &19(20,19) \\ \hline\hline
8(4)                   &      &17(17,17) &16(16,16)  &20(20,19) &19(19,18)   &19(19,20)  &19(19,20)  \\
12(4)                  &      &18(17,17) &16(16,16)  &20(20,20) &23(22,18)   &19(20,20)  &18(19,19) \\ \hline
\end{tabular}
\end{table}

In summary, when $k \ge 2$, the iteration counts of PCG method based on $(B^{(1)}_{k,\psi})^{-1}$ are nearly independent of the mesh size, overlapping width, jump coefficient distribution and jump range, moreover, $(B^{(1)}_{k,\psi})^{-1}$ is more efficient than $(B^{(1)}_{k,\bar{\psi}})^{-1}$.

\section{TL-OS preconditioners for plane wave discretization of Helmholtz equations}\label{Helmholtz}
Consider the following Helmholtz equations
\begin{eqnarray}\label{chap4-model equation}
\left\{
\begin{array}{rcll}
 -\Delta u  - \kappa^2 u &=&  0,        &in~ \Omega,\\
 (\partial_{\bf n} + i \kappa )u &=&  g,    &    on~ \partial \Omega,
\end{array}
\right.
\end{eqnarray}
where $i=\sqrt{-1}$, $\partial_{\bf n}$ and $\kappa=\frac{\omega}{c}>0$ are separately the imaginary unit, the outer normal derivative, and the wave number, in which $\omega$ and $c$ are separately called the angular frequency and the wave speed, $g\in L^2(\partial \Omega)$(refer to \cite{HY2014}). 

Let $\mathcal{T}_h$ be a uniform hexahedral(or quadrilateral) isometric partition of $\Omega$(see Fig. \ref{fig-dd-Th} for illustration in two dimension), i.e.
$\bar{\Omega} = \bigcup_{k=1}^{N_h} \bar{\tau}_k,$
where $N_h$ and $h$ are separately the number of elements and the mesh size, the elements $\{\tau_k\}$ satisfy that $\tau_m \cap \tau_l = \emptyset, m \neq l$. Denote $\gamma_{kj} = \partial \tau_k \cap \partial \tau_j, k\neq j$ and $\gamma_k = \partial \tau_k \cap \partial \Omega$.

Define the following plane wave function space(\cite{HY2014})
\begin{eqnarray}\label{Vp-space}
V_{h,2} =span\{\varphi_{m,l}:~1 \le l \le p, 1\le m \le N_h\},
\end{eqnarray}
where $p$ is the number of degrees of freedom in each element and
\begin{equation*}
\varphi_{m,l}(\bx)=
\left\{
\begin{array}{ll}
y_{m,l}(\bx)     &\bx \in \bar{\tau}_m\\
     0          &\bx \in \Omega \backslash \bar{\tau}_m
\end{array}
\right. ,
\end{equation*}
and refer to \cite{HY2014} for the expression of $y_{m,l} (l=1,\cdots,p)$.

Then we get the PWLS formulation associated with Helmholtz problem \eqref{chap4-model equation}: find $u \in V_{h,2}$ such that
\begin{eqnarray*}
a^{(2)}(u,v) =  \mathcal{L}(v),~\forall v \in V_{h,2},
\end{eqnarray*}
where
\begin{eqnarray}\nonumber
  a^{(2)}(u,v) &=& \sum_{j\neq k} (\alpha_{kj}\int_{\gamma_{kj}} (u_k-u_j)\cdot\overline{(v_k-v_j)}ds
+\beta_{kj}\int_{\gamma_{kj}}
({\partial_{{\bf n}_k}u_k+\partial_{{\bf n}_j}u_j})\cdot\overline{(\partial_{{\bf n}_k}v_k+\partial_{{\bf n}_j}v_j)}ds)
  \nonumber\\
    &&~~~~~~~~+ \sum_{k=1}^{N_h} \nu_k \int_{\gamma_k} ((\partial_{\bf n}+i\kappa_k ){u}_k) \cdot \overline{(\partial_{\bf n}+i \kappa_k){v}_k}ds, \label{115-1}\\
\mathcal{L}(v) &=& \sum_{k=1}^{N_h} \nu_k \int_{\gamma_k} g\cdot\overline{(\partial_{\bf n}+i \kappa_k ){v}_k}ds.\nonumber
\end{eqnarray}
Here ${\bf n}_l$ is the unit normal vector of $\tau_{l}$, $\kappa_l=\kappa|_{\tau_l}$, $\overline{\diamond}$ denotes the complex conjugate of $\diamond$, and refer to  \cite{HL2017} for the expression of the Lagrange multipliers $\alpha_{kj}, \beta_{kj}, \nu_k$.

The discrete system of \eqref{chap4-model equation} associated with $V_{h,2}$ can be written as
\begin{eqnarray}\label{sys-helmholtz}
A_{h,2} u_{h,2}=f_{h,2},~~~~u_{h,2}\in V_{h,2}.
\end{eqnarray}

\begin{figure}[t]
\centering
\subfigure[]{ \label{fig-dd-Th}
\includegraphics[width=0.8in]{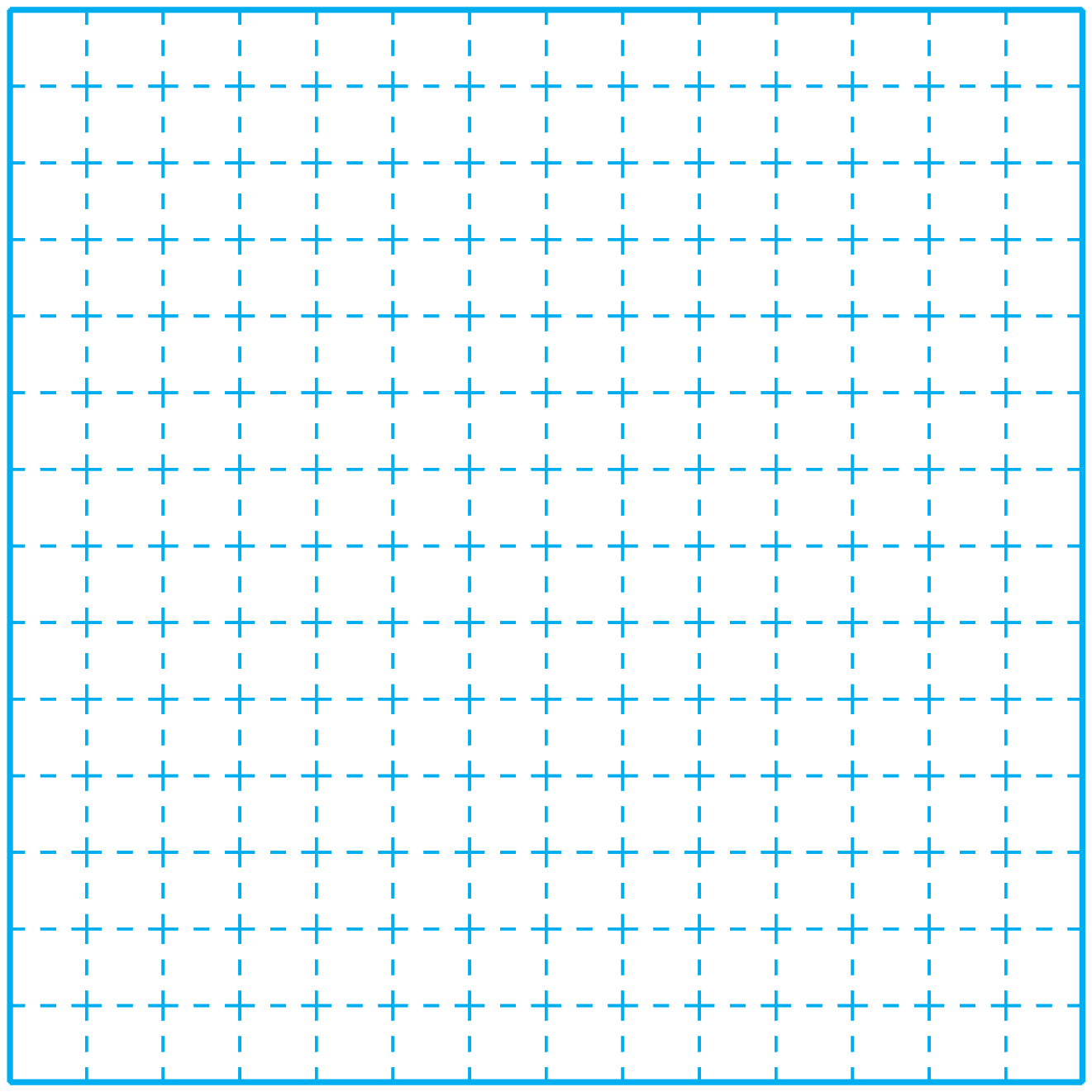}}
\hspace{0.001in}
\subfigure[]{ \label{fig-dd-1}
\includegraphics[width=0.8in]{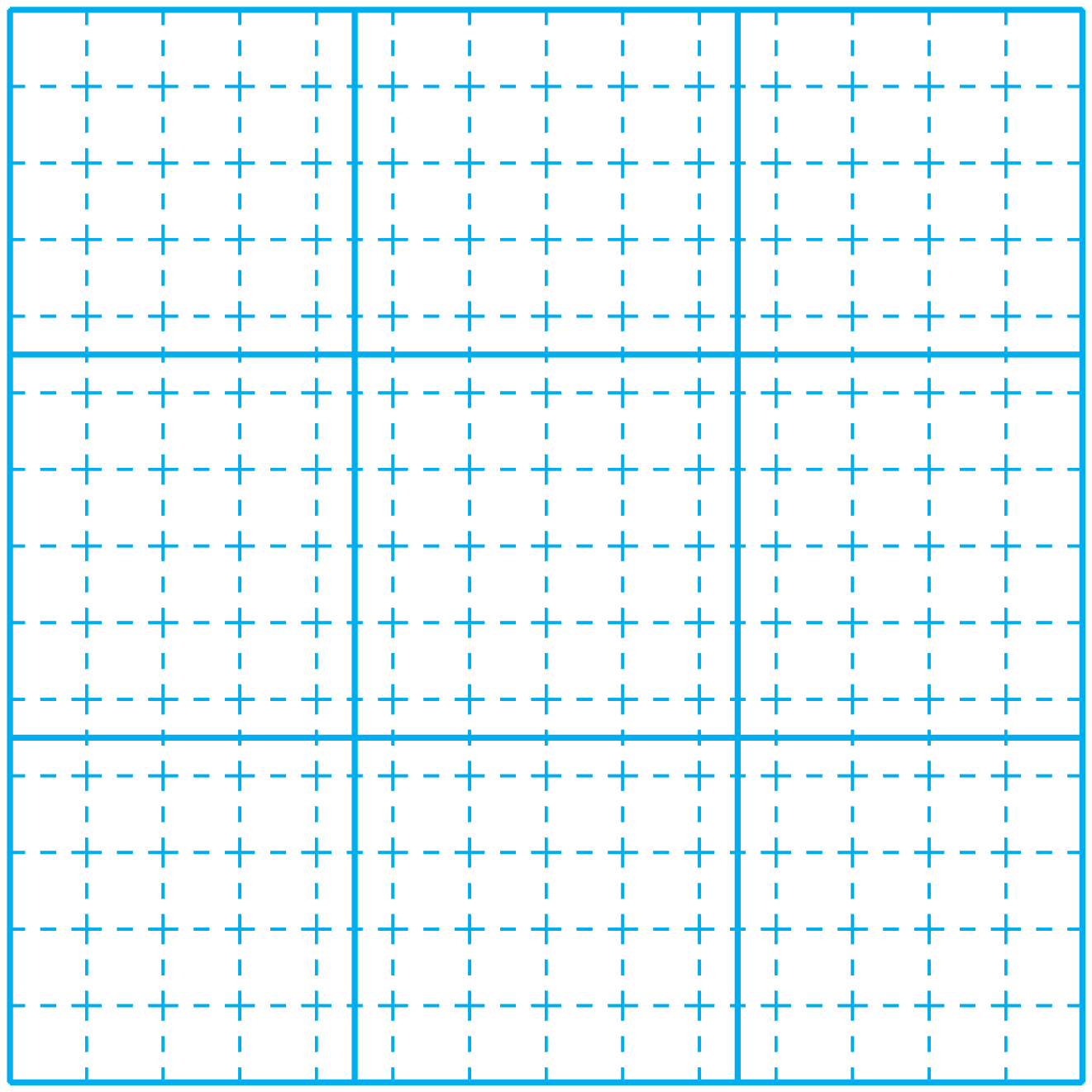}}
\hspace{0.001in}
\subfigure[]{ \label{fig-dd-2}
\includegraphics[width=0.8in]{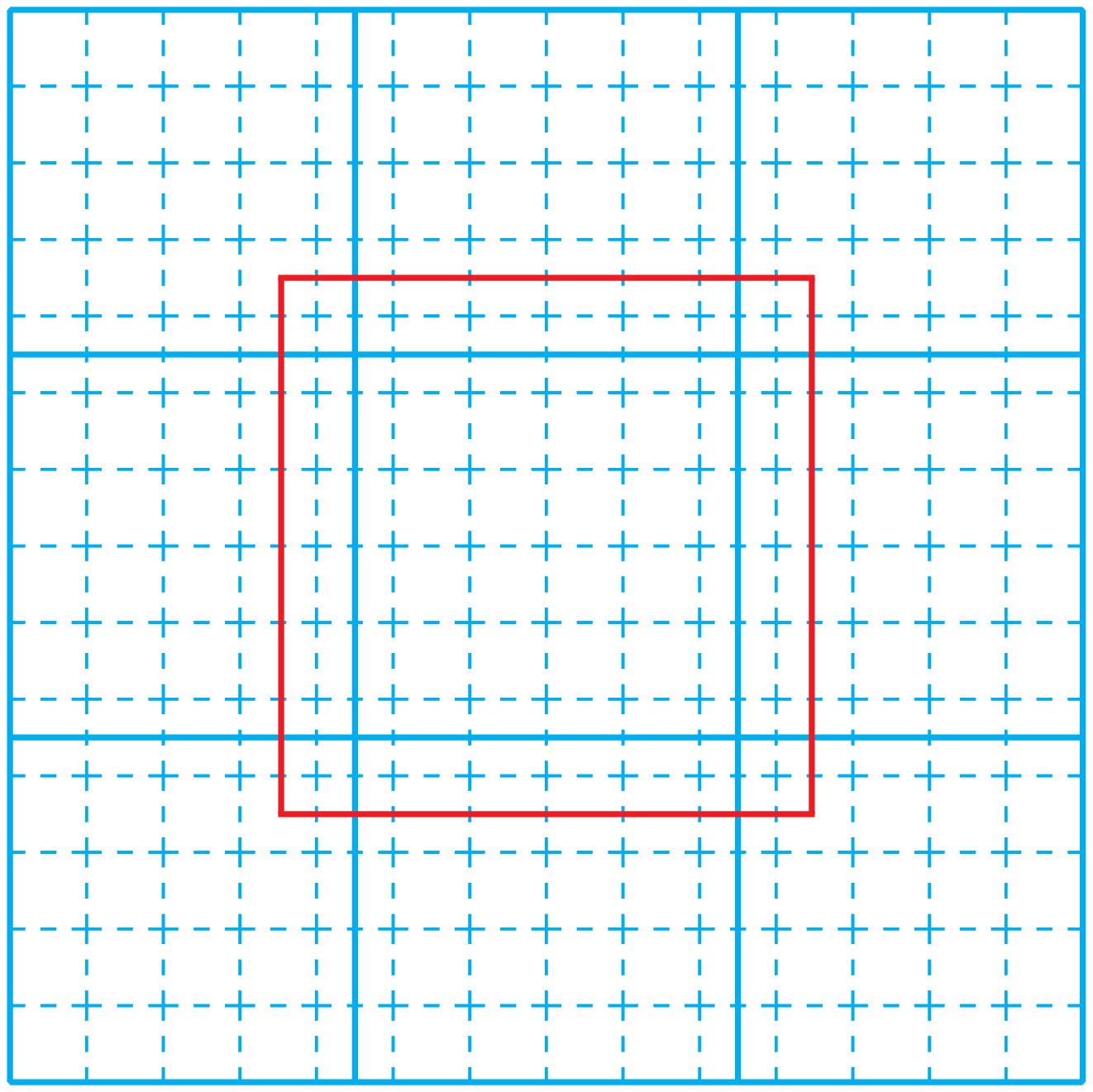}}
\hspace{0.001in}
{\caption{(a)~$\mathcal{T}_h$, (b)~$\mathcal{T}_h=3(4)$, $\Omega_i$: the square with solid line boundary (c)~$\Omega'_i:=\Omega^{(i)}_{l,h}$ with $l=1$: the square with red solid line boundary.}
\label{Fig:submesh-PWLS}}
\end{figure}

We introduce a special nonoverlapping domain decomposition $\{\Omega_i\}_{i=1}^N$ of $\mathcal{T}_h$(\cite{PWS2018}) where the interface pass through the elements(Fig. \ref{fig-dd-1}), and the overlapping subdomain $\Omega'_i:=\Omega^{(i)}_{l,h}$ is obtained by the similar way mentioned in subsection \ref{framework-TL-OS}(Fig. \ref{fig-dd-2}).
Using the above partition and $a^{(2)}(\cdot,\cdot)$ defined by \eqref{115-1}, we can construct an overlapping Schwarz preconditioner similar to $B^{-1}_s$ in subsection \ref{framework-TL-OS}. Numerical experiments show that the number of iterations is much more dependent on the model and mesh parameters than second order elliptic problem if we apply PCG method based on this preconditioner to solve \eqref{sys-helmholtz} directly. Therefore, how to design a TL-OS preconditioner based on the framework in subsection \ref{framework-TL-OS} is more urgent for PWLS system.

Let $\{\theta_i(x)\}^N_{i=1}$ be a partition of unity which are linear finite element functions defined on the dual partition of $\mathcal{T}_h$. Fig. \ref{1d-theati} shows some $\theta_i(x)$ in one dimension, where ``$\cdot$'' and ``$\times$'' represents the partition node and the midpoint of the element(also called the dual partition nodes) respectively. $\Omega_i$ and $\Omega'_i$ are separately the $i$-th nonoverlapping subdomain and overlapping subdomain with $l=1$. Let $\mathcal{Y}_h$ be the set of the dual partition nodes, then define $\theta_i(\bx)$ with its values at any $\bx\in \mathcal{Y}_h$ being
\begin{eqnarray*}
\theta_i(\bx)=
\left\{
\begin{array}{lll}
\frac{1}{|\mathcal{N}_{\bx}|}, & if ~ \bx\in \Omega'_i ~or~\bx\in \partial\Omega'_i\cap \partial\Omega \\
0,& otherwise
\end{array}
\right.,
\end{eqnarray*}
where $|\mathcal{N}_{\bx}|$ denote the number of elements in $\mathcal{N}_{\bx}=\{j: \bx\in \Omega'_j~or~\bx\in \partial\Omega'_j\cap \partial\Omega\}$. Direct calculation shows that these $\{\theta_i(\bx)\}_{i=1}^N$ satisfy \eqref{theta-def}.


Let $\mathcal{S}^{(i)}$ be the index set defined by \eqref{S-i-M}, and define the sesquilinear and Hermitian positive definite functional
\begin{eqnarray}\label{s-PWLS}
s^{(2)}_i(u,w)&=&\sum_{\gamma_{kj}\subset \bar{\Omega}_i}\beta_{kj}\int_{\gamma_{kj}} \sum\limits_{l\in \mathcal{S}^{(i)}}|\nabla \theta_l(\bx)|^2
(u_k \overline{w}_k+u_j\overline{w}_j ) ds\nonumber\\
&+&\sum_{\gamma_k \in \partial\Omega \cap \partial\bar{\Omega}_i } \nu_k \int_{\gamma_k} \sum\limits_{l\in \mathcal{S}^{(i)}}|\nabla \theta_l(\bx)|^2 u_k \overline{w}_kds.
\end{eqnarray}

Now we confirm that the sesquilinear form $s^{(2)}_i(\cdot,\cdot)$ defined by \eqref{s-PWLS} satisfies the ASSUMPTION \ref{assump-1-strong}. In fact, for any $u\in V_{h,2}$, using \eqref{115-1} and \eqref{s-PWLS}, we have
\begin{eqnarray*}
a^{(2)}_i(\theta_i u,\theta_i u)
   &\le&  \sum_{\gamma_{kj}\subset \Omega_i} (2\beta_{kj}\int_{\gamma_{kj}}
|\partial_{{\bf n}_k}u_k+\partial_{{\bf n}_j}u_j|^2+|u_k\partial_{{\bf n}_k}\theta_i +u_j\partial_{{\bf n}_j}\theta_i|^2ds
  \\ 
   &+& \alpha_{kj}\int_{\gamma_{kj}} |u_k-u_j|^2ds)+ \sum_{\gamma_k \subset \partial\Omega \cap \partial\Omega_i } 2\nu_k \int_{\gamma_k} |u_k\partial_{\bf n}\theta_i|^2+|(\partial_{\bf n}+i\kappa_k){u}_k|^2ds\\
&\le& 4 \big(a^{(2)}_i(u,u)
  + \sum_{\gamma_{kj}\subset \Omega_i}\beta_{kj}\int_{\gamma_{kj}}
|\nabla\theta_i|^2(|u_k|^2+|u_j|^2) ds\\
&+&\sum_{\gamma_k \subset \partial\Omega \cap \partial\Omega_i } \nu_k \int_{\gamma_k} |\nabla\theta_i|^2|u_k|^2ds\big)\le 4 (a^{(2)}_i(u,u)+s^{(2)}_i(u,u))
\end{eqnarray*}
namely, $s^{(2)}_i(\cdot,\cdot)$ satisfies the ASSUMPTION \ref{assump-1-strong} in subsection \ref{framework-TL-OS}. $\Box$

Using $s^{(2)}_i(\cdot,\cdot)$ and $a^{(2)}_i(\cdot,\cdot)$, we can introduce the corresponding coarse space $V_{0,2}:=V_0$ defined by \eqref{zjb-def}.
In an analogous way to the definitions of $A_{j,1}, A_{0,1}$, $\Pi_{j,1}$ and $\Pi_{0,1}$, let $A_{j,2}, A_{0,2}$, $\Pi_{j,2}$ and $\Pi_{0,2}$ be the operators corresponding to the plane wave function space $V_{h,2}$ and the sesquilinear form $a^{(2)}(\cdot,\cdot)$, then by formula \eqref{B-def-new}, we construct the TL-OS preconditioner for system \eqref{sys-helmholtz} as follows
\begin{eqnarray*}
B^{-1}_2 = \Sigma_{j=1}^N \Pi_{j,2} A_{j,2}^{-1}\Pi_{j,2}^*+\Pi_{0,2} A_{0,2}^{-1}\Pi_{0,2}^*.
\end{eqnarray*}

For the following three typical models(\cite{PWS2018}\cite{HZ2016}), we investigate the performance of PCG method with $B^{-1}_2$ as preconditioner for solving the system \eqref{sys-helmholtz}.

\begin{model}\label{chap4-Example-1}
Consider model problem \eqref{chap4-model equation} where $\Omega = (0,2) \times (0,1)$ and ~$c \equiv 1$, the exact solution of the problem can
be expressed as
\begin{eqnarray*}
u_{ex} = \cos(12 \pi y)(A_1 e^{-i \omega_x x} + A_2 e^{i \omega_x x}),
\end{eqnarray*}
where $\omega_x = \sqrt{\omega^2 - (12 \pi)^2}$, and coefficients $A_1$ and $A_2$ satisfy
\begin{eqnarray*}
& \left(
\begin{array}{cc}
\omega_x  & -\omega_x \\
(\omega - \omega_x)e^{-2i \omega_x} & (\omega + \omega_x)e^{2i \omega_x}
\end{array}
\right) \left(\begin{array}{c}
A_1 \\
A_2
\end{array}\right)
= \left(\begin{array}{c}
-i\\
0
\end{array}\right).
\end{eqnarray*}
\end{model}

\begin{model}\label{chap4-Example-2}
Consider model problem \eqref{chap4-model equation} where $\Omega = (0,7200)\times(0,3600)$, $\Gamma_{d} = \Gamma_{n} = \emptyset$, $g = x^2 + y^2$, and the wave speed $c$ is defined by(see Fig. ~\ref{chap4-fig-piecewise-constant})
\begin{equation*}
c(x,y) = \left\{\begin{array}{ll}
1800 & \mbox{if}~y \in [0,1200)\\
3600 & \mbox{if}~y \in [1200,2400)\\
5400 & \mbox{if}~y \in [2400,3600]
\end{array}\right.~~x\in [0,7200].
\end{equation*}
\end{model}

\begin{model}\label{chap4-Example-3}
Consider MODEL \ref{chap4-Example-2} where $c(x,y)$ is chosen randomly from $[1500$,$5500]$ for each grid cell(see Fig. \ref{chap4-fig-random}).
\end{model}

\begin{figure}[t]
\centering
\subfigure[]{\label{chap4-fig-piecewise-constant}
\includegraphics[width=1.0in]{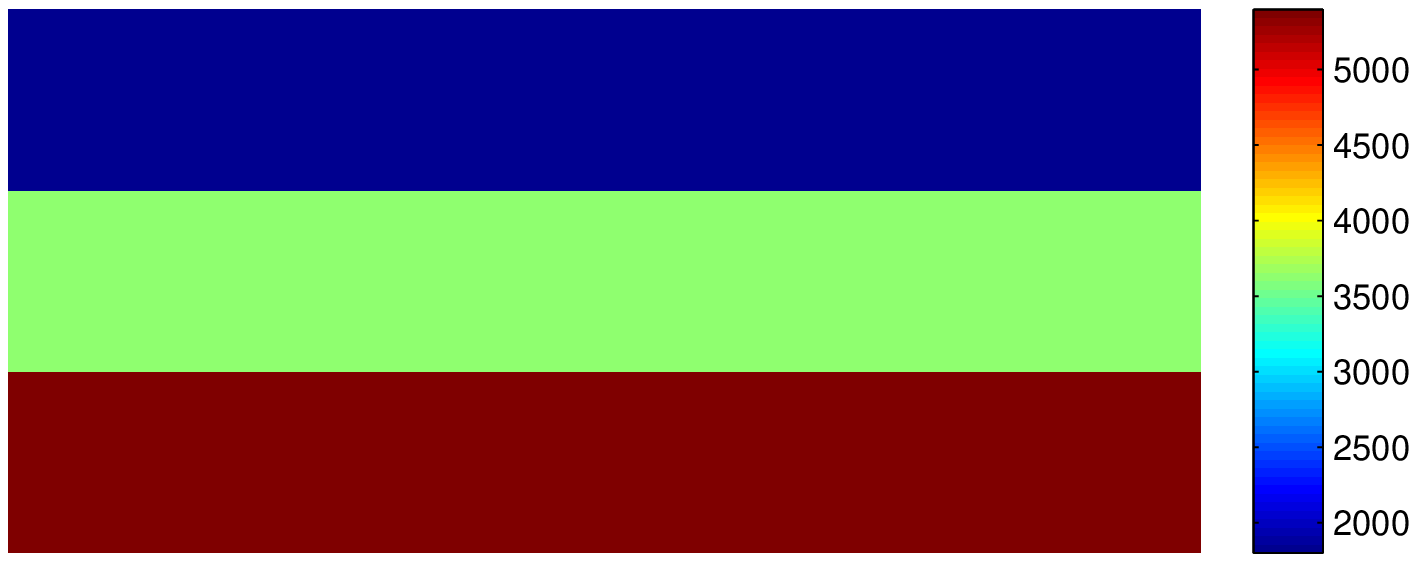}}
\hspace{0.01in}~~~~~~
\subfigure[]{ \label{chap4-fig-random} 
\includegraphics[width=1.0in]{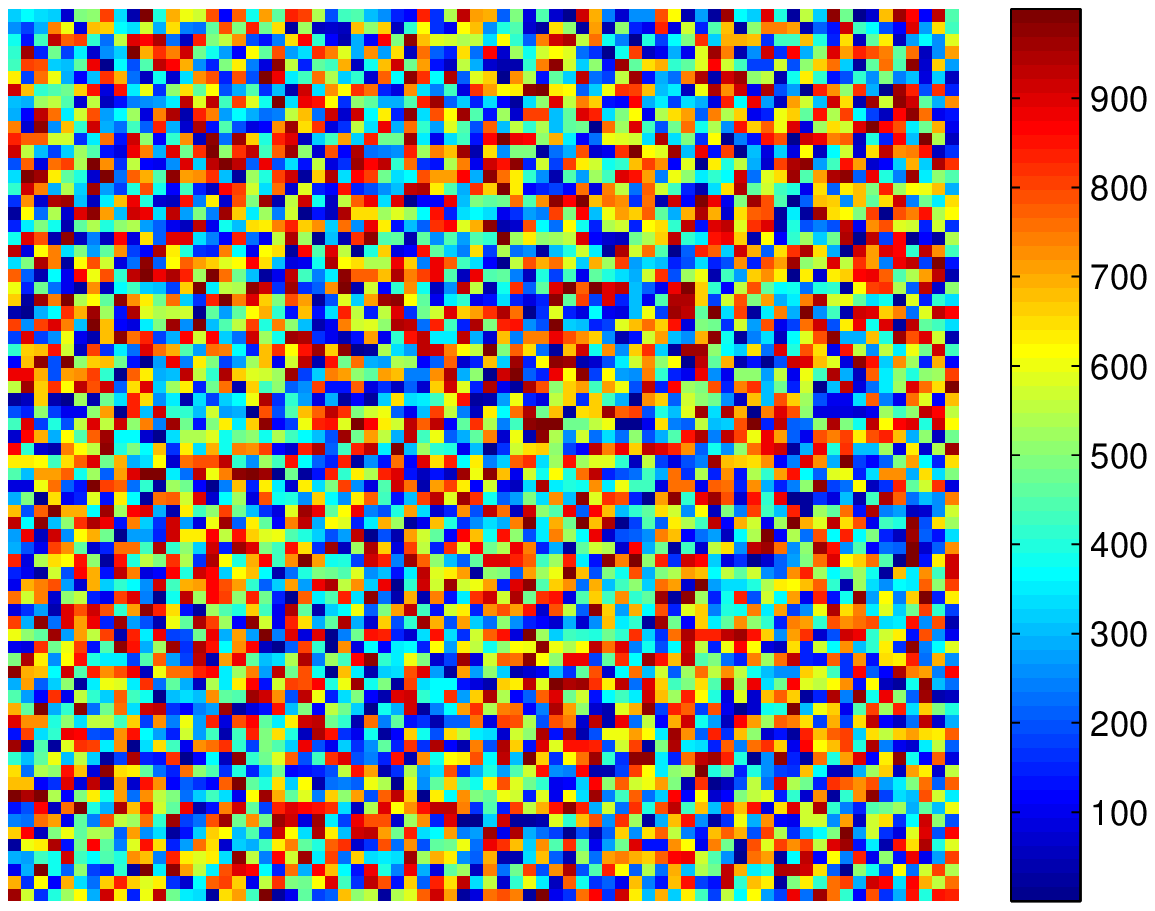}}
\hspace{0.01in}
{\caption{(a)~wave speed $c$ in MODEL \ref{chap4-Example-2}, (b)~wave speed $c$ in MODEL \ref{chap4-Example-3}}
\label{Fig:submesh-PWLS-c}}
\end{figure}

Denote $\mathcal{T}_h=n(m)$ where $m$ represents the number of complete elements in each subdomain for each direction which is different from subsection \ref{Poisson-TL-OS},  e.g. Fig. \ref{fig-dd-1} shows the case of $m=4$.
Set $\Lambda=1+\log(H/h+2)$ and $tol=10^{-5}$. In the numerical experiments, the selection principle of $n(m)$ and $p$ is: keeping $\omega h$ is constant and increasing $p$ or decreasing $h$ appropriately to make the relative error less than $5\times 10^{-3}$.

%
%

The left and right tables in TABLE \ref{101} show the number of iterations as $l$ and $n(m)$(with appropriate $p$) changing respectively. We find from the two tables that for a given model, the iterations of PCG method based on the designed TL-OS preconditioner $B^{-1}_2$ are weakly dependent on these factors.

\begin{table}
  \label{101}
  \tiny
  \caption{Left: $\omega=20\pi, n(m)=4(8), p=9$. Right: $\omega=20\pi, l=1$}
  \centering
\begin{tabular}{|c|c|c|c|}
\hline
$l$   & MODEL \ref{chap4-Example-1} &MODEL \ref{chap4-Example-2} &MODEL \ref{chap4-Example-3} \\
\hline
1   &16     & 13  & 25 \\
2   &14     & 11  & 26  \\
3   &14     & 11  & 28  \\
4   &13     & 10  & 28  \\
\hline
\end{tabular}~
\begin{tabular}{|c|c|c|c|c|c|}
  \hline
  $p$  & $n(m)$& MODEL \ref{chap4-Example-1} &MODEL \ref{chap4-Example-2} &MODEL \ref{chap4-Example-3} \\
 \hline
  9  &4(8)  &16     & 13  & 25  \\
     &4(12) &15    & 12  & 19  \\\hline
 10  &4(8) &18     & 10  & 18  \\
     &5(8) &16     & 12  & 24  \\
\hline
\end{tabular}
\end{table}

Based on economical framework proposed in subsection \ref{framework-economical},  we can give the following two economical TL-OS precondionters
\begin{eqnarray}\label{precond-ec-helmholtz}
(B^{(2)}_{k,\zeta})^{-1} = \Sigma_{j=1}^N \Pi_{j,2} A_{j,2}^{-1}\Pi_{j,2}^*+\Pi^{k,\zeta}_{0,2} (A^{k,\zeta}_{0,2})^{-1}(\Pi^{k,\zeta}_{0,2})^*,~~\zeta=\psi, \bar{\psi},
\end{eqnarray}
where $\Pi^{k,\zeta}_{0,2}: V^{k,\zeta}_{0,2}\rightarrow V_{h,2} (\zeta=\psi, \bar{\psi})$ are  identical lifting operators, the coarse space operator $A^{k,\zeta}_{0,2}: V^{k,\zeta}_{0,2}\rightarrow V^{k,\zeta}_{0,2}$ satisfies
$(A^{k,\zeta}_{0,2}u,v)=a^{(2)}(u,v), \forall u,v\in V^{k,\zeta}_{0,2}$,
and $V^{k,\zeta}_{0,2}\subset V_{h,2}$ here is the corresponding economical coarse space.

In the following, we only consider PCG method with $(B^{(2)}_{k,\psi})^{-1}$ as preconditioner to solve the PWLS system \eqref{sys-helmholtz}. The iteration counts are listed in TABLE \ref{ec-A-l-1} for MODEL \ref{chap4-Example-1}-MODEL \ref{chap4-Example-3} when $\omega$, $k$, $l$, $n(m)$ and $p$ are varying. In addition, we also consider the non-homogeneous case of Helmholtz equation, for which the right hand side of the first equation in \eqref{chap4-model equation} is $f=\left(1-\omega^{2}\right) \omega x \cos y$, $\Omega = (0,1)^2$, $g = (\partial_{\bf n} + i \omega) u_{ex}$
and the exact solution $u_{ex}(x, y)=\omega x \cos y+y \sin (\omega x)$.
The results are listed in TABLE \ref{non-home-ec-A}. The values outside and inside brackets in TABLE \ref{ec-A-l-1} and TABLE \ref{non-home-ec-A} correspond to $l=1,2$, respectively. We can see that the number of iterations is weakly dependent on the model and mesh parameters when $k\ge 2$.
\begin{table}[h]
\scriptsize\centering\caption{$l=1, 2$}
\label{ec-A-l-1}\vskip 0.1cm
\begin{tabular}{{|c|ccc|cc|cc|}}\hline
\multirow{2}{*}{$k$}& \multicolumn{3}{c|}{{\tiny $n(m)=12(4),p=10,\omega=20\pi$}} & \multicolumn{2}{c|}{\tiny $n(m)=8(7),p=13,\omega=40\pi$}  & \multicolumn{2}{c|}{\tiny $n(m)=10(11),p=14,\omega=80\pi$} \\\cline{2-8}
                     &  MODEL \ref{chap4-Example-1} &MODEL \ref{chap4-Example-2} &MODEL \ref{chap4-Example-3} & MODEL \ref{chap4-Example-1} &MODEL \ref{chap4-Example-2} & MODEL \ref{chap4-Example-1} & MODEL \ref{chap4-Example-2}\\       \hline
2                   & 17(16) &23(21) &30(43) &18(16) &14(11) &21(19) &14(13)  \\
3                   & 18(14) &15(11) &30(41) &19(17) &11(9) &22(20) &10(8) \\ 
4                   & 18(14) &15(11) &32(41) &19(16) &12(8) &22(19) &10(8)  \\ \hline
\end{tabular}
\end{table}



\begin{table}[H]
\footnotesize\centering\caption{$l=1, 2$}
\label{non-home-ec-A}\vskip 0.1cm
\begin{tabular}{{|c|c|c|c|}}\hline
$k$& \multicolumn{1}{c|}{{\scriptsize $\omega=4\pi,p=7,n(m)=5(6)$}} & \multicolumn{1}{c|}{\scriptsize $\omega=8\pi,p=9,n(m)=5(6)$}  & \multicolumn{1}{c|}{\scriptsize $\omega=8\pi,p=9,n(m)=8(7)$} \\\cline{1-4}
1                   &15(17) &18(17) &19(21) \\  
2                   &15(15) &18(15) &19(18) \\
3                   &15(15) &18(15) &19(17) \\ \hline
\end{tabular}
\end{table}

\section{Conclusion}

A coarse space with lower computational complexity is designed for general Hermitian positive and definite discrete variational problems, with which we established the algorithmic frameworks of the corresponding TL-OS and two kinds of economical TL-OS preconditioners. Based on certain assumptions, we established the condition number estimate theory of the preconditioning systems. As the application of the above frameworks, firstly, we designed TL-OS and economical TL-OS preconditioners for second order elliptic equation in three dimension, and estimated the condition number bounds. Moreover, we found that the economical preconditioning estimation is independent of the jump range of the coefficient under special jump distribution. In addition, the computational efficiency of the existing precondtioner is greatly improved due to the sparsity of the coefficient matrix of the new coarse basis formula. Secondly, TL-OS and economical TL-OS preconditioners are designed for the PWLS discrete system of Helmholtz equation. Numerical results show that the PCG methods based on the proposed preconditioners have good stability.

In future work, we will do more in-depth research on the TL-OS preconditioning estimate theory of PWLS discrete system of Helmholtz equation, and extend the work of this paper to other complex models. In addition, we will study the corresponding multi-level algorithm to reduce the dimension of coarse space and enhance the calculation scale and efficiency.

\end{document}